\documentclass[11pt]{amsart}

\usepackage{amssymb,graphicx, amsmath, amsthm, enumitem,MnSymbol, array,mathalpha}
\usepackage{calrsfs}
\usepackage{wrapfig}
\usepackage{graphicx}
\graphicspath{ {./images/} }

\DeclareMathAlphabet\mathbfcal{OMS}{cmsy}{b}{n}

\usepackage{todonotes}

\usepackage{floatflt}

\usepackage{hyperref}

\usepackage{mathbbol}

 \usepackage[mathscr]{eucal}

\usepackage[font=small]{caption}

\def\R{\mathbb R} 

\def\c{\mathscr}

\def\R{{\Rappa}}

\def\PP{\Phi}

\def\PP{\overline{\PP}}

\def\P{{\mathbb{P}}}

\def\C{{\bf C}}

\begin{document}

\newtheorem{theorem}{Theorem}[section]
\newtheorem{lemma}[theorem]{Lemma}
\newtheorem{proposition}[theorem]{Proposition}
\newtheorem{corollary}[theorem]{Corollary}
\newtheorem{problem}[theorem]{Problem}
\newtheorem{construction}[theorem]{Construction}

\theoremstyle{definition}
\newtheorem{defi}[theorem]{Definitions}
\newtheorem{definition}[theorem]{Definition}
\newtheorem{notation}[theorem]{Notation}
\newtheorem{remark}[theorem]{Remark}
\newtheorem{example}[theorem]{Example}
\newtheorem{question}[theorem]{Question}
\newtheorem{comment}[theorem]{Comment}
\newtheorem{comments}[theorem]{Comments}

\newtheorem{discussion}[theorem]{Discussion}

\renewcommand{\thedefi}{}

\long\def\alert#1{\smallskip{\hskip\parindent\vrule%
\vbox{\advance\hsize-2\parindent\hrule\smallskip\parindent.4\parindent%
\narrower\noindent#1\smallskip\hrule}\vrule\hfill}\smallskip}

\def\ff{\frak}
\def\tf{torsion-free}
\def\Spec{\mbox{\rm Spec }}
\def\Proj{\mbox{\rm Proj }}
\def\hgt{\mbox{\rm ht }}
\def\type{\mbox{ type}}
\def\Hom{\mbox{ Hom}}
\def\rank{\mbox{rank}}
\def\Ext{\mbox{ Ext}}
\def\Tor{\mbox{ Tor}}
\def\Rer{\mbox{ Ker }}
\def\Max{\mbox{\rm Max}}
\def\End{\mbox{\rm End}}
\def\xpd{\mbox{\rm xpd}}
\def\Ass{\mbox{\rm Ass}}
\def\emdim{\mbox{\rm emdim}}
\def\epd{\mbox{\rm epd}}
\def\repd{\mbox{\rm rpd}}
\def\ord{\mbox{\rm ord}}

\def\htt{\mbox{\rm ht}}

\def\DD{{\mathcal D}}
\def\EE{{\mathcal E}}
\def\FF{{\mathcal F}}
\def\GG{{\mathcal G}}
\def\HH{{\mathcal H}}
\def\II{{\mathcal I}}
\def\LL{{\mathcal L}}
\def\MM{{\mathcal M}}
\def\PP{{\mathcal P}}

\def\R{\mathbb{R}}

\title{Rectangles conformally inscribed in lines}

\author{Bruce Olberding} 
\address{Department of Mathematical Sciences, New Mexico State University, Las Cruces, NM 88003-8001}

\email{bruce@nmsu.edu}

\author{Elaine A.~Walker}
\address{1801 Imperial Ridge, Las Cruces, NM 88011}

\email{miselaineeous@yahoo.com}

\begin{abstract}   A parallelogram is conformally inscribed in  four lines in the plane if it is inscribed in a scaled copy of the configuration of four lines. We describe the geometry of the three-dimensional Euclidean space whose points are  the parallelograms  conformally inscribed in sequence in these four lines. In doing so, we describe the flow of inscribed rectangles by introducing a compact model of the rectangle inscription problem. 

 \end{abstract}


\thanks{\today}

\maketitle

\section{Introduction}

  This  article continues  recent work from \cite{OW,OW2,Sch,Sch2, Tup}, in   
   which the  geometric and algebraic properties of rectangles inscribed in lines in the plane are studied. 
  Describing such rectangles is the key step to describing rectangles inscribed in polygons, which in turn is related to the rectangle peg problem of finding rectangles inscribed in simple closed curves; see for example    \cite{GL, Kak, Mas, Mey, Sch2}. 
  To describe rectangles inscribed in lines, we need only consider four lines at a time. 
Denote by $\C$ a configuration consisting of two pairs of lines $A,C$ and $B,D$, where  not all four lines  are parallel. 
 A rectangle is {inscribed} in $\C$  if its vertices lie on $A,B,C,D$ in either clockwise or counterclockwise order.  
  Each  rectangle inscribed in $\C$ lies on a path of inscribed rectangles, either
 a ``slope path''  
   parametrized by  the slope of the rectangles or an ``aspect path'' parametrized by aspect ratio \cite{OW2,Sch}. (The images of these two paths can coincide, as they do in Figure~1.) 
   As they travel either path, the rectangles eventually grow without bound, yet as they do so their slopes and aspect 
 ratios   converge to a pair of numbers that do not occur as the slope or aspect ratio of any rectangle inscribed in $\C$. 
 
  \begin{figure}[h] \label{diagonalspar}
 \begin{center}
 \includegraphics[width=0.9\textwidth,scale=.09]{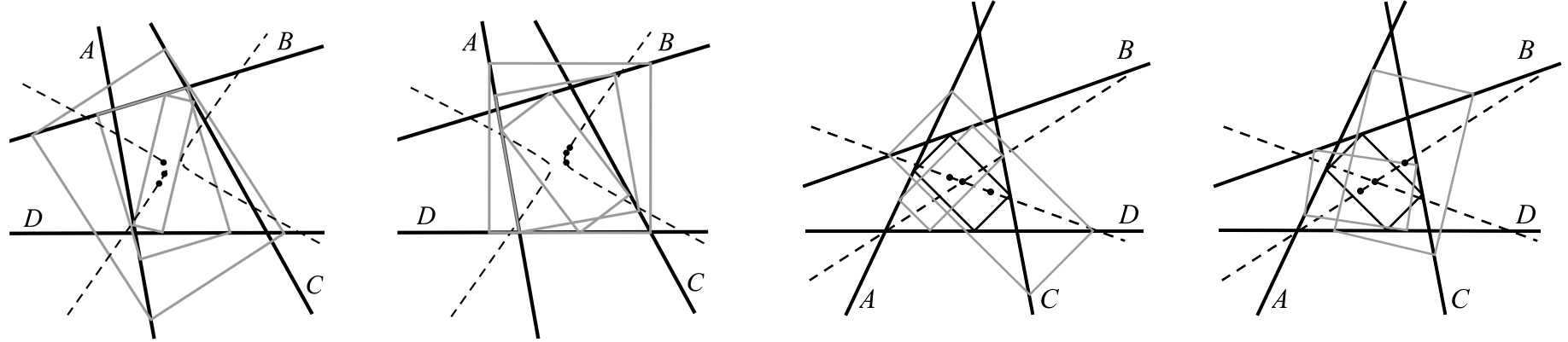} 
 \end{center}
 \caption{In the figures at left, the slope and aspect paths trace the same hyperbola and contain the same rectangles. (This is always the case if $\C$ is non-degenerate \cite[Theorem~8.1]{OW2}.)  The hyperbola  is the set of centers of the rectangles inscribed in $\C$.  The figures at right illustrate the case in which the hyperbola is degenerate. Here the slope path traces one line and the aspect path the other. 
   }
\end{figure}
 
 These two numbers, a missing slope and a missing aspect ratio, suggest a missing rectangle, a rectangle inscribed at infinity. In \cite[Section 3]{OW2}, this is dealt with by viewing the  inscribed rectangles as points in $\R^8$, with two coordinates for each vertex of the rectangle. Taking the projective closure  of the set of inscribed rectangles in $\R^8$ finds the missing rectangle as a point at infinity. The vertices of a rectangle can be read off from the  coordinates of this rectangle at infinity, but instead of a rectangle inscribed in the four lines, it is
an equivalence class of rectangles  
  inscribed in the configuration $\C$ viewed from infinity, in which all four lines go through the same point.  
 
It is proved in \cite[Theorem~3.2]{OW2} that when viewed  as points in $\R^8$, the rectangles inscribed in $\C$ constitute  a line or a   planar hyperbola in $\R^8$, which
if the rectangles at infinity are included
  becomes either a  simple closed  curve of rectangles  in the projective closure $\P^8$ of $\R^8$ or a pair of lines. This insertion of a missing rectangle at infinity is sufficient for the purposes in \cite{OW2}, but with this approach part of the geometry of the rectangle inscription problem breaks at infinity, since  at infinity, and only at infinity, the configuration switches from $\C$ to the configuration in which all four lines are translated to the origin. 
  Also, the 
   rectangle at infinity remains indefinite since 
 it is really an equivalence class of rectangles rather a single rectangle through which the other inscribed  rectangles pass. 
 
 These last two observations 
   are the starting point of this article and suggest the need for 
  a kind of curve selection lemma which  chooses from each equivalence  class in $\P^8$ a rectangle in such a way that the set of selected rectangles is a simple closed curve of rectangles inscribed not only in $\C$ but in scaled copies of $\C$ also. Viewing $\P^8$ as lines through the origin in $\R^9$, this can be accomplished using the fact that the 8-dimensional unit sphere is a double cover of $\P^8$. The desired simple closed curve of rectangles is then found on the unit sphere. 
   
   However, it's possible to describe the data that determines the rectangles much more efficiently than as points in $\R^8$ and  to give a less generic   explanation for the existence of a simple closed curve of rectangles that scale to the rectangles inscribed in $\C$.    To do so, in Section 2
      %
we define a notion of conformally inscribed parallelograms as those parallelograms that are inscribed in scaled copies of $\C$, and we show that the set of all such parallelograms inscribed in $\C$  forms a three-dimensional Euclidean space $\Pi$ with inner product determined by the diagonals of the parallelograms. 
This space has an orthonormal basis of three canonically chosen parallelograms that determine principal axes for $\Pi$. 
 This is helpful because the equations that govern the inscribed rectangles are long and unwieldy, 
 but they  become simple and easy to work with in the coordinates induced by this orthonormal basis. The previous complexity of the equations is hidden  in the coordinates for the parallelograms in the orthonormal basis.

Among the conformally inscribed rectangles are the {\it unit} rectangles, those conformally inscribed rectangles that have norm $1$ with respect to the inner product of $\Pi$. Every rectangle inscribed in the four lines can be represented by a unit rectangle inscribed in a scaled copy of the configuration of the four lines. We show that in the space $\Pi$ of conformally inscribed parallelograms, the unit rectangles lie on an intersection of two cylinders and thus form an algebraic space curve. This curve 
is the union of two simple closed analytic curves, and each of these curves gives a compact solution to the rectangle inscription problem in the sense that these rectangles comprise a flow that 
after projection accounts for all rectangles inscribed in $\C$ as well as those at infinity.


All this can be made more concrete by
 reinterpreting a model for the rectangle inscription problem from \cite{OW}, where it is shown that locating the {\it centers} of the rectangles inscribed in $\C$ is equivalent to describing the intersection of two hyperbolically rotated cones in $\R^3$. 
In the model, which we call in Section~7 the cone model for $\C$, the
intersection of the two cones is a space curve that projects orthographically  to a curve in $\R^2$ consisting of the rectangle centers. This model also omits the rectangles at infinity: the space curve heads to infinity as the centers of the rectangles in the plane head to infinity. 

However, as we discuss in Section~7, we may view the cone model as a chart in the three-dimensional projective space $\P^3$.  If we exchange the plane in the cone model in which the cone apexes reside with the plane at infinity, 
we bring the missing rectangles at infinity into view.  The hyperbolically rotated cones now have apex at infinity in this new model, which we call in Section~6 the cylinder model, and so  become elliptical cylinders. The intersection of these two cylinders is a compact space curve which consists of the centers of the unit rectangles.
 The parallelograms at infinity lie in a plane in the cylinder model, and so having the flow of inscribed rectangles pass through a rectangle at infinity is simply a matter  of having these rectangles pass through this plane. What makes this possible is a scaling dimension $w$ so that passing through the rectangles at infinity amounts to passing from positively scaled copies of $\C$ to negatively scaled copies.  It is this last feature, having another side of infinity, that is missing from the cone model for $\C$ and which makes possible the approach taken in this article. 
 
 We have use Maple$^{\rm TM}$ to assist with calculations and graphics.


\section{Conformally inscribed parallelograms}

When dealing with   parallelograms inscribed in  lines in the plane, we may  reduce to the case of four lines. If the four lines are not all parallel, 
we can relabel and rotate the lines in order to reduce further to the 
 following standing assumption for the article. (See  \cite{OW2,Sch2} for discussion of this reduction.) The restriction  that the line $B$ go through the point $(0,1)$ can always be achieved by scaling the configuration $\C$.

\begin{quote} {\bf Standing assumption}. 
Throughout this article we work with two pairs of lines $A,C$ and $B,D$  in the real plane $\R^2$ such that the four lines do not meet in a single point;
 $B$ goes through the point $(0,1)$;
and $C$ and $D$  meet in the origin and only in the origin. The configuration consisting of the two pairs of lines $A,C$ and $B,D$ is denoted $\C$.  
\end{quote}

We  work with equations for these four lines, and for this we use the following notation.   

\begin{notation}
We denote by $m_A,m_B,m_C,m_D,b_A $
the real numbers 
for which the equations defining the lines $A,B,C,D$ are 
$$A: \: y = m_Ax+b_A \:\:\: B: y=m_Bx+1 \:\:\:
C: y=m_Cx \:\:\: D: y=m_Dx.$$
 We write $m_{AB}$ for $m_A-m_B$, $m_{BC}$ for $m_B - m_C$, etc..
\end{notation}

   \begin{notation} \label{first notation}

Let $w\in \R$. For a line $L$ defined by an equation of the from $y = mx+b$, we denote by $L(w)$ the line in $\R^2$ defined by $y=mx+bw$. With $A,B,C,D$ the lines from the standing assumption, the line $A(w)$ is defined by $y=m_Ax+b_Aw$ and the line $B(w)$  by $y=m_Bx+w$. Since $C$ and $D$ pass through the origin,  $C(w) = C$ and $D(w) = D$.  We denote by $\C(w)$  the resulting  configuration consisting of the pairs $A(w),C(w)$ and $B(w),D(w)$. 
Thus for $w \ne 0$, $\C(w)$ is a scaled copy of $\C$. The special case where $w=0$ is the configuration $\C(0)$ of the four lines $A(0),B(0),C(0),D(0)$ through the origin.  This is the   configuration $\C$ viewed from infinity. 
\end{notation}


\begin{definition} 
A parallelogram $P$ is {\it inscribed in $\C$} if the vertices of $P$ lie in sequence, either clockwise or counterclockwise, on the lines $A,B,C,D$.  Among  inscribed parallelograms, we include the degenerate ones also,  those whose sides all lie on the same line, or even whose sides consist of a single point, namely the origin in the configuration $\C(0)$. 
A parallelogram $P$ in the plane $\R^2$ is {\it conformally inscribed} in $\C$ if  
$P$ can be inscribed in $\C(w)$ for some $w \in \R$. In this case, we say $w$ is the scale of $P$.  
 If $w =0$, then $P$ is {\it inscribed at infinity} for $\C$.  
We denote by 
 $\Pi$ the set of parallelograms conformally inscribed in $\C$.
\end{definition}





We give the set $\Pi$ of conformally inscribed parallelograms the structure of a  Euclidean space. 
Define the sum of two parallelograms to be the parallelogram that results from adding vertices coordinatewise.  
 If one of the parallelograms is inscribed in the configuration $\C(w_1)$ and the other in $\C(w_2)$, then the sum of the two parallelograms is inscribed in the configuration $\C(w_1+w_2)$.  Similarly, if $\lambda \in \R$, then multiplying the coordinates of each vertex of a parallelogram inscribed in $\C(w)$ by $\lambda$
results in a parallelogram inscribed in $\C(\lambda w)$. 
The degenerate parallelogram all of whose vertices are the origin is the additive identity for $\Pi$. 
 With these operations, $\Pi$ is a vector space. 

We define an inner product on $\Pi$ in terms of the diagonals of the parallelograms in $\Pi$. It is convenient to view these diagonals as vectors, as in the next definition. 

\begin{definition} 
\label{diagonal}
Let $P$ be a parallelogram in  $\Pi$, and let $w$ be the scale of $P$.  Denote by   $(x_L,y_L) \in L(w)$, $L \in \{A,B,C,D\}$, the vertices of $P$. 
 The $AC$  and $BD$ {\it diagonal vectors} of $P$ are, respectively, $$P_{AC} =(x_A-x_C,y_A-y_C),  \:\:
 P_{BD}=(x_B-x_D,y_B-y_D).$$
 The {\it side vectors} of $P$ are $P_{AB}=(x_A-x_B,y_A-y_B)$ and $P_{BC}=(x_B-x_C,y_B-y_C)$. 
 \end{definition}

\begin{definition} \label{inner def}
We define an inner product on $\Pi$ for each $P,Q \in \Pi$ by 
\begin{eqnarray*}
\left<P,Q\right> &= & 
P_{AC} \cdot Q_{AC} + P_{BD} \cdot Q_{BD} \\
&=& 2P_{AB} \cdot Q_{AB} + 2 P_{BC} \cdot Q_{BC},
\end{eqnarray*}
and we denote by $\| P\|$ the norm induced by this inner product, i.e. $\| P\|= \left<P,P\right>^{\frac{1}{2}}.$


\end{definition}

An argument similar to that for the parallelogram law shows that the second equality in the definition holds.  
Thus $\| P\|^2$ is the sum of the squared lengths of the  diagonals of $P$, which in turn is twice the sum of the squared lengths of the sides of $P$.
  That the bilinear map $\left<,\right>$ defines an inner product 
follows from the fact that if $\left<P,P\right> =0$, then 
$P_{AB} =P_{BC} = {\bf 0}$, which can only happen if 
%
$P$ is the degenerate parallelogram that is simply  the origin in $\R^2$.  Thus  $\left<P,P\right> =0$ if and only if $P$ is the zero element in $\Pi$.

\begin{theorem} \label{ortho} With the inner product in Definition~$\ref{inner def}$, 
 $\Pi$ is a Euclidean space  having 
 an orthonormal basis $U,V,T$ of $\Pi$ such that
 $U$ and $V$ are degenerate parallelograms and 
 $U_{AC} =V_{BD}={\bf 0}$ and $T_{AC} \cdot V_{AC} =
 T_{BD} \cdot U_{BD} =0.$ $($See Figure $2$.$)$
%
 The parallelograms  $U$ and $V$ are  at infinity if and only if
  at least one of the pairs  $A,C$ or $B,D$ consists of  parallel
lines.

  \end{theorem}

  \begin{figure}[h] \label{Basic special}
 \begin{center}
 \includegraphics[width=0.65\textwidth,scale=.03]{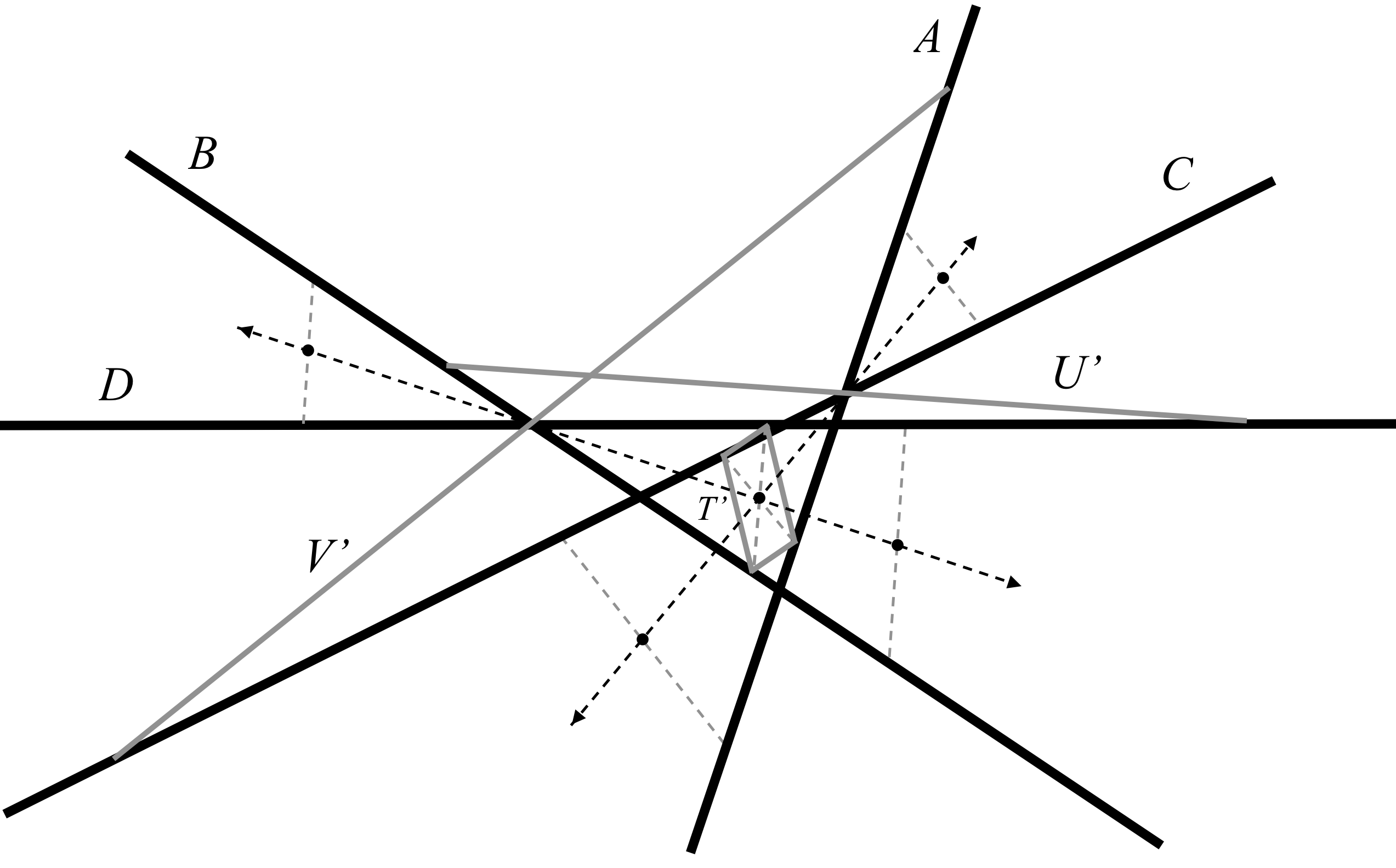} 
 \end{center}
 \caption{With $U,V,T$ as in Theorem~\ref{ortho}, 
 the  degenerate unit parallelogram $U$ scales to the degenerate parallelogram $U'$ (in gray). Similarly, $V$ scales to    $V'$ and $T$ scales to the non-degenerate parallelogram  $T'$. 
The center of $U'$ is the intersection of $A$ and $C$ and the center of $V'$ is the intersection of $B$ and $D$.   
 The center of $T'$ lies on the intersection of the two dotted lines in the figure, which are the lines $L_U$ and $L_V$ from Remark~\ref{LU}. }
 %
 \end{figure}

\begin{proof}
We first show $\dim(\Pi)=  3.$ 
Let $w \in \R$. By \cite[Lemma 4.3]{OW2}, a quadrilateral inscribed in $\C(w)$ with vertices $(x_L,y_L) \in L$, $L \in \{A(w),B(w),C(w),D(w)\}$, 
is a parallelogram 
 if and only if $x_D = x_A+x_B-x_C$ and 
 $x_C = 
 {(m_{DC})}^{-1}(m_{AD}x_A+ 
m_{DB}x_B+(b_{A}-1)w).  $
(By our standing assumptions, $m_C \ne m_D$, so this is a simple consequence of the fact that a quadrilateral is a parallelogram if and only if the midpoints of its two  diagonals  coincide.) 
Since a parallelogram is specified by its vertices, each choice of $x_A,x_B,w \in \R$ determines  a unique parallelogram in $\Pi$ according to these equations, and every parallelogram in $\Pi$ can be specified by the choice of $x_A,x_B,w$. Thus $\dim(\Pi) = 3$.

Next, observe that if $P$ and $Q$ are parallelograms in $\Pi$ such that $P_{AC}$ and $Q_{BD}$ are the zero vector, then $P$ and $Q$ are orthogonal: 
$\left<P,Q\right>^2=P_{AC}\cdot Q_{AC} + P_{BD}\cdot Q_{BD} =0$.
Thus, if $U$ and $V$ are parallelograms such that the $AC$ diagonal of $U$ and the  $BD$ diagonal of  $V$ have length $0$, then $U$ and $V$ are orthogonal. 

We construct the parallelogram $U$ first.  
If neither pair $A,C$ or $B,D$ consists of parallel lines, set
 $$x_A = -\frac{b_A}{m_{AC}}, \: \: x_B= -\frac{2b_Am_{CD}+m_{AC}}{m_{AC}m_{BD}}, \:\:w=1$$
Then, using the expressions above involving $x_C$ and $x_D$, 
these choices for $x_A,x_B,w$ define a parallelogram $P$ in $\Pi$ with vertices $(x_L,y_L) \in L$, for $L \in \{A,B,C,D\}$. Also, a calculation shows 
 $x_A
=x_C$ and $y_A=y_C$, so that $P$ defines a parallelogram in $\Pi$ whose $AC$ diagonal vector is ${\bf 0}$. Since $B$ does not go through the origin, the $BD$ diagonal  vector of $P$ is not ${\bf 0}$. Multiplying $P$ by $\|P\|^{-1}$, we obtain a parallelogram $U$ in $\Pi$ with norm 1 whose $AC$ diagonal has length $0$.

On the other hand, if $B$ and $D$ are parallel or $A$ and $C$ are parallel, 
%
then with
$$x_A = \frac{m_{BD}}{2m_{CD}}, \:\: x_B=1, \:\: w=0,$$ we obtain a parallelogram with $x_A=x_C$ and $ y_A=y_C$, and this parallelogram can be rescaled  to a parallelogram $U$ that has norm $1$. 

Next, to define $V$, assume first that 
 neither pair $A,C$ or $B,D$ is parallel. Set
$$ x_A = -\frac{b_Am_{BD}-2m_{CD}}{m_{BD}m_{AC}}, \:\: x_B = -\frac{1}{m_{BD}}, \:\:w=1.$$ As above, this defines a parallelogram $P$ with vertices $(x_L,y_L) \in L$,  $L \in \{A,B,C,D\}$. A calculation shows $x_B=x_D$ and $y_B=y_D$. Since the lines $A,B,C,D$ do not all go through the origin, it cannot be that $x_A=x_C$ and $y_A=y_C$. Thus $\|P\| >0$, and as above, rescaling gives a parallelogram $V$  with norm 1 whose $BD$ diagonal has length $0$. 

On the other hand,  if $m_A = m_C$, then  the values $x_A=1,x_B=0,w=0$ define a parallelogram that after rescaling gives the parallelogram $V$.  If instead $m_A \ne m_C$ and $m_B = m_D$, then 
$$x_A =- \frac{2m_{CD}}{m_{AC}}, \:\: x_B = 1, \:\: w=0$$ 
defines a parallelogram that after rescaling gives $V$. 

 The construction of $U$ and $V$  shows that one of the parallelograms $U,V$ is at infinity (i.e., has scale $w=0$) if and only if both are; if and only if at least one of the pairs  $A,C$ and $B,D$ consists of parallel lines. Also, since $U_{AC} ={\bf 0}$ and $V_{BD}={\bf 0}$, the parallelograms $U$ and $V$ are orthogonal. 


Finally, since $\Pi$ has dimension  $3$, the orthogonal complement of the subspace of $\Pi$ spanned by $U$ and $V$ has dimension $1$. Choose a parallelogram $T$ in this subspace for which $\|T\| = 1$.  Since $U_{AC}$ is the zero vector and $T$ is orthogonal to $U$, it follows that $T_{BD} \cdot U_{BD} =0$. Similarly, $T_{AC}\cdot V_{AC}=0$. 
\end{proof}



\begin{remark} \label{LU}
Suppose neither pair $A,C$ or $B,D$ consists of parallel lines. The construction of the parallelogram $T$ in the proof of Theorem~\ref{ortho} leads to the following method for finding~$T$. 
Let $L_U$ be the median curve for the lines $A$ and $C$ along a line perpendicular to the degenerate parallelogram $U$; that is, let $L_U$ be
 the line consisting of the midpoints of the line segments that join $A$ and $C$ and are perpendicular to the degenerate parallelogram $U$. Similarly,  let $L_{V}$ be the line   of midpoints of the  segments perpendicular to  $V$ that join $B$ and $D$. If $L_{U}$ and $ L_{V}$ intersect at a point then this point is the center of the copy of $T$ scaled to a parallelogram inscribed in $\C$. This is illustrated in Figure~2. 
Otherwise, if $L_U$ and $L_V$ are parallel (see Figure~3), then $T$ is at infinity. 
\end{remark}

In light of the theorem, we  introduce coordinates for the Euclidean space $\Pi$ in terms of the basis $U,V,T$.  

\begin{notation} For $u,v,t \in \R$, we denote by $(u,v,t)$ the parallelogram $uU+vV+tT$ in $\Pi$. We also denote by $w_U,w_V,w_T$ the scales of the parallelograms $U,V,T$. 
\end{notation}

Since $U,V,T$ is an orthonormal basis, for each $u,v,t\in \R$, we have 
 $$\|uU+vV+tT\|=\sqrt{u^2+v^2+t^2}.$$  Thus in the coordinates $(u,v,t)$, the set of {\it unit parallelograms} in $\Pi$,  i.e., those with norm $1$, is a sphere defined by $u^2+v^2+t^2=1$.  
 
   By Theorem~\ref{ortho},  one of the parallelograms $U,V$ is at infinity if and only if both are; if and only if at least one of the pairs  $A,C$ and $B,D$ consists of parallel lines. 
 In this case, the parallelogram $T$ is not at infinity since otherwise the linearly independent vectors $U,V,T$ all lie in the two-dimensional subspace of $\Pi$ consisting of the parallelograms at infinity, a contradiction. 
 On the other hand, if $U$ and $V$ are not at infinity (and so neither pair of lines $A,C$ or $B,D$ consists of parallel lines), then the orthogonal complements $U^\perp$ and $V^\perp$ of $U$ and $V$ are distinct planes through the origin. These two planes meet in a line which contains $T$. This line, and hence $T$, can be at infinity or not. See Figure~3 for an example where $T$ is at infinity.

\begin{figure}[h] \label{Basic special}
 \begin{center}
 \includegraphics[width=0.65\textwidth,scale=.03]{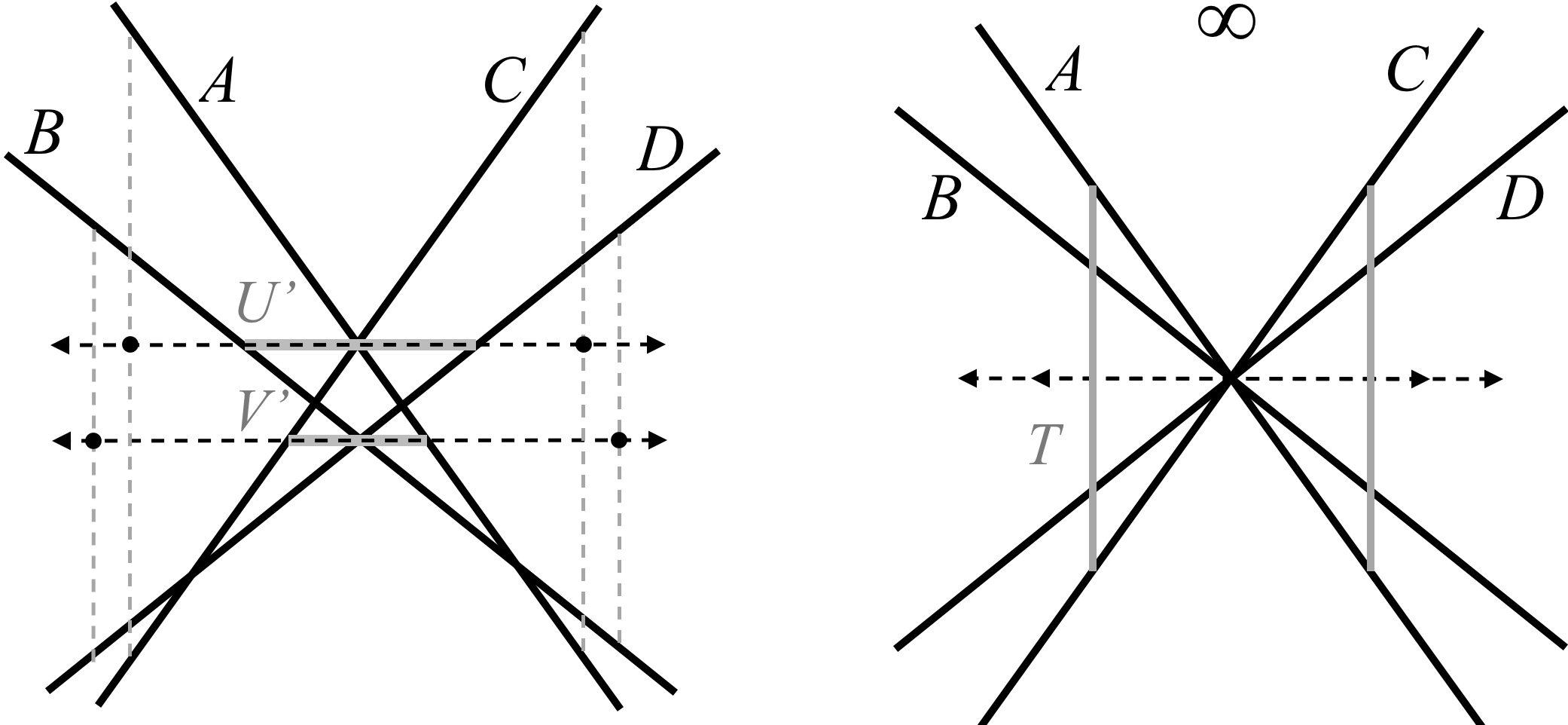} 
 \end{center}
 \caption{The figure at right is the configuration at left viewed from infinity.  With $U,V,T$ as in Theorem~\ref{ortho}, 
  $T$ is degenerate and at infinity. The degenerate parallelograms $U$ and $V$ scale to 
the degenerate parallelograms $U'$ and $V'$ (in gray), respectively. The dotted arrows are the lines $L_U$ and $L_V$ from Remark~\ref{LU}. The fact that they are parallel is why $T$ is at infinity.}
 \end{figure}

 The parallelograms  inscribed in $\C$  can be distinguished among the parallelograms conformally inscribed in $\C$, according to  the next corollary.  Recall that $w_U,w_V,w_T$ are the scales of $U,V,T$, respectively.

\begin{corollary} \label{w cor} For each $w \in \R$, the set of parallelograms $(u,v,t)$ inscribed in $\C(w)$ is  a plane in $\Pi$ defined by $w = uw_U+vw_V+tw_T$. 
This plane goes through the origin if and only if $w =0$, and so the parallelograms at infinity form a two-dimensional subspace of $\Pi$. 
\end{corollary} 

\begin{proof} 
For each $w \in \R$, the parallelograms inscribed in $\C(w)$ are the parallelograms $(u,v,t)$ that have scale $w$. If $P$ and $Q$ are two parallelograms in $\Pi$ of scales $w_P$ and $w_Q$, then the scale of $P+Q$ is $w_P+w_Q$ since $P \in \C(w_P)$ and $Q \in \C(w_Q)$ imply $P+Q \in \C(w_P+w_Q)$.  
Thus the scale $w$ of the parallelogram $(u,v,t)$ is 
 $w = uw_U+vw_V+tw_T.$ 
The plane in $\Pi$ defined by this equation goes 
  through the origin if and only if $w =0$. The parallelograms of scale $0$ are precisely the   parallelograms at infinity, so the corollary follows.
\end{proof}

Since $\Pi$ is a vector space,  we may consider its projective space $\P \Pi$  whose points are the lines in $\Pi$ through the origin.  The next corollary is a straightforward consequence of the fact from Theorem~\ref{ortho} that $\Pi$ is a three-dimensional vector space. 

\begin{corollary} \label{proj}
The projective space $\P \Pi$ of $\Pi$ is the projective closure of the plane of parallelograms in $\Pi$ that are inscribed in $\C$. 
Each point in  $\P \Pi$ is  the set of all the scaled copies of a parallelogram inscribed in $\C$ or a parallelogram at infinity for $\C$. \qed
\end{corollary}

The next theorem shows that the inner product on $\Pi$ is a natural choice, being induced as it by the usual dot product on $\R^4$. This dot product involves the diagonal vectors of the parallelogram, which when extracted 
 from a parallelogram $P$ in $\Pi$ 
results in the loss of the location of the parallelograms since only the shape and orientation is preserved by the diagonal vectors.  Thus, on their own, the diagonal vectors
 determine $P$ only up to translation.
 However, with the data  of the original configuration  $\C$ in which $P$ was conformally  inscribed, we can recover the location of $P$. In particular, there is only one parallelogram in $\Pi$ that 
has a given set of diagonal vectors. Similarly, there is only one parallelogram in $\Pi$ that has a given set of  side vectors. 

 \begin{theorem} \label{L} The linear transformation that sends a parallelogram $P$ in $\Pi$ to its diagonal vectors $(P_{AC},P_{BD})$ in $\R^2 \times \R^2$ is an isometry onto its image. Thus
the diagonal vectors of  $P $ uniquely determine $P$, as do the side vectors.  
\end{theorem}

\begin{proof} Let $M$ be the $4 \times 3$ matrix 
$$\begin{bmatrix} U_{AC} & V_{AC} & T_{AC} \\
U_{BD} & V_{BD} & T_{BD} \\
\end{bmatrix}$$
Since $U,V,T$ is an orthonormal basis, this matrix has rank~$3$. If $P$ is a parallelogram with coordinates $(u,v,t)$ and vertices $(x_L,y_L) \in L$, for $L \in \{A(w),B(w),C(w),D(w)\}$, then  since $P = uU+vV+tT$, we have 
$$M \begin{bmatrix} 
u \\
v \\
t\\
\end{bmatrix} = \begin{bmatrix} P_{AC} \\
P_{BD} \\
\end{bmatrix}.$$
Since $M$ has rank $3$, $(u,v,t)$ is the unique solution to this matrix equation. The matrix $M$ is the representation with respect to the basis $U,V,T$ of the linear transformation that sends a parallelogram to its diagonal vectors, and so this transformation is injective. The transformation is an isometry onto its image since   $\|P\|^2= P_{AC}\cdot P_{AC} + P_{BD} \cdot P_{BD}$ and the latter expression is an application of the standard norm for $\R^4$.  

To prove the statement about side vectors, observe that $P_{AC}=P_{AB}+P_{BC}$ and $P_{BD} = P_{BC}+P_{CD}= P_{BC}-P_{AB}$.  Thus, with $I$ the $2 \times 2$ identity matrix, we have 
$$
\begin{bmatrix} P_{AC} \\
P_{BD} \\
\end{bmatrix} =
\begin{bmatrix}  I & I \\
I & -I \\
\end{bmatrix} 
\begin{bmatrix} P_{BC} \\
P_{AB} \\
\end{bmatrix}.$$
This gives  an invertible linear transformation from the space of diagonal vectors to the space of side vectors, and since the representation of $P$ by its diagonal vectors is unique, so is the representation of $P$ by its side vectors. 
\end{proof}



\section{The geometry of the space of parallelograms}

In this section we describe the   geometry of the set of rectangles in the Euclidean space~$\Pi$. Since $\Pi$ contains not only the rectangles inscribed in $\C$ but also those conformally inscribed in $\C$, there is flexibility in choosing rectangles to represent, up to scale, the rectangles inscribed in $\C$. A natural way to do this is to select the {\it unit rectangles}, those rectangles having norm $1$ in $\Pi$, and we do this in this section.  

The parallelograms $U,V,T$ from Theorem~\ref{ortho}, along with the coordinates $(u,v,t)$ that represent the parallelograms in $\Pi$,  play a main role  in our approach in this section. The three lines in $\Pi$ that pass through the origin and each of these three parallelograms form orthogonal axes for $\Pi$ and make the geometric properties of the set of unit rectangles a straightforward matter. For example, these axes and the invariants of $T$ introduced in the following notation are all that are needed to specify the unit rectangles, as we show in Theorem~\ref{rect cor}. 

\begin{notation} \label{lambda} With $T$ the parallelogram given by Theorem~\ref{ortho}, 
 let 
 $\lambda$ be the squared length of the $AC$ diagonal of $T$, and $\mu$ be the squared length of the $BD$ diagonal; that is, \begin{center} $\lambda = T_{AC} \cdot T_{AC} $ \: and \:  $\mu =T_{BD}\cdot T_{BD} $. \end{center} Since $\|T\|=1$, we have $\lambda+\mu=1$.  

\end{notation}

 It is possible that $\lambda=0$ or $\mu = 0$.  For example, $\lambda =0$  if $A=C$ or $A$ and $C$ are parallel and the intersection of $B$ and $D$ is equidistant from lines $A$ and $C$. See Figure 4 for the latter example.

\begin{figure}[h]  
 \begin{center}
 \includegraphics[width=0.7\textwidth,scale=.03]{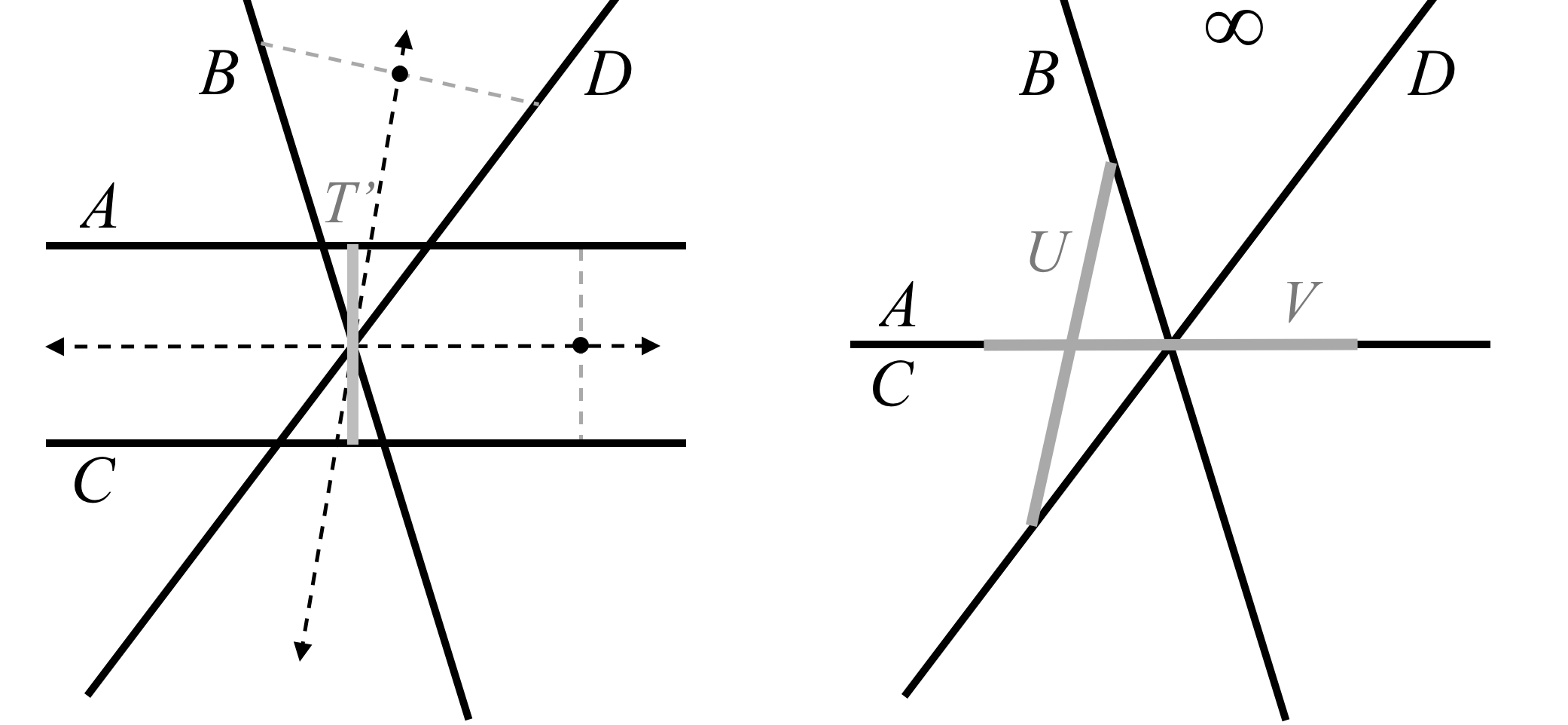} 
 \end{center}
 \caption{Since $A$ and $C$ are parallel in the configuration ${\bf C}$ at left, the parallelograms $U$ and $V$ from Theorem~\ref{ortho} are at infinity, which is represented at right. 
  The parallelogram $T$ is also degenerate and it scales to the inscribed parallelogram $T'$ whose  center lies on the two lines $L_U$ and $L_V$ (shown with dashes) from Remark~\ref{LU}. 
 In this example, $\lambda =0$ since the $AC$ diagonal of $T$ has length~$0$.}
 \end{figure}

\begin{lemma} \label{cylinders} In the coordinates $(u,v,t)$, the  set of unit rectangles in $\Pi$ is an algebraic curve given by the intersection of the surfaces defined by
\begin{center} $2u^2+2\mu t^2=1$ \: and \: $2v^2+2\lambda t^2=1.$
\end{center}
Each of these surfaces is  an elliptical cylinder or pair of parallel planes $($see Figure~$5)$. The first surface
 is the set of parallelograms in $\Pi$ whose $BD$ diagonal has length $1/2$ and the second  is the set of parallelograms whose $AC$ diagonal has length $1/2$. 
\end{lemma} 

\begin{proof} 
First we show that for a parallelogram $P$ in $\Pi$ with coordinates $(u,v,t)$,
\begin{center}  $P_{AC} \cdot P_{AC} = v^2+\lambda t^2$ \: and \:  $P_{BD} \cdot P_{BD} = {u^2+\mu t^2}$.\end{center}
Since  $V_{BD}=\bf 0$ and $\|V\|=1$, we have $V_{AC} \cdot V_{AC}=1$.  
Therefore, since $U_{AC}={\bf 0}$  and  $P = uU+vV+tT$,  Theorem~\ref{ortho} implies
\begin{eqnarray*}
P_{AC} \cdot P_{AC} &= &  v^2(V_{AC}\cdot V_{AC})+  2vt(V_{AC} \cdot T_{AC}) +t^2(T_{AC} \cdot T_{AC}) \\
&= &  v^2+ \lambda t^2.
\end{eqnarray*}
This shows  the equation $2v^2+2\lambda t^2=1$ defines the set of parallelograms whose $AC$ diagonal has squared length $1/2$. A similar argument shows $P_{BD} \cdot P_{BD} = u^2+\mu t^2$ and hence $2u^2+2\mu t^2=1$  defines the set of parallelograms whose $BD$ diagonal has squared length $1/2$. 
 The parallelograms that lie in this intersection of these two quadrics have diagonals of squared length $1/2$, and so these parallelograms are precisely the  unit rectangles in $\Pi$.


Finally, if $\lambda \ne 0$ and $\mu \ne 0$, the equations  $2v^2+2\lambda t^2=1$  and  $2u^2+2\mu t^2=1$
 define two elliptical cylinders that meet on the unit sphere in $\Pi$.   Otherwise, if $\lambda =0$, then $2v^2+2\lambda t^2=1$ defines a pair of parallel planes in $\Pi$ and similarly for $\mu =0$ and the second surface.  
 \end{proof}



\begin{figure}[h] \label{cylindersgraphic}
 \begin{center}
 \includegraphics[width=0.75\textwidth,scale=.03]{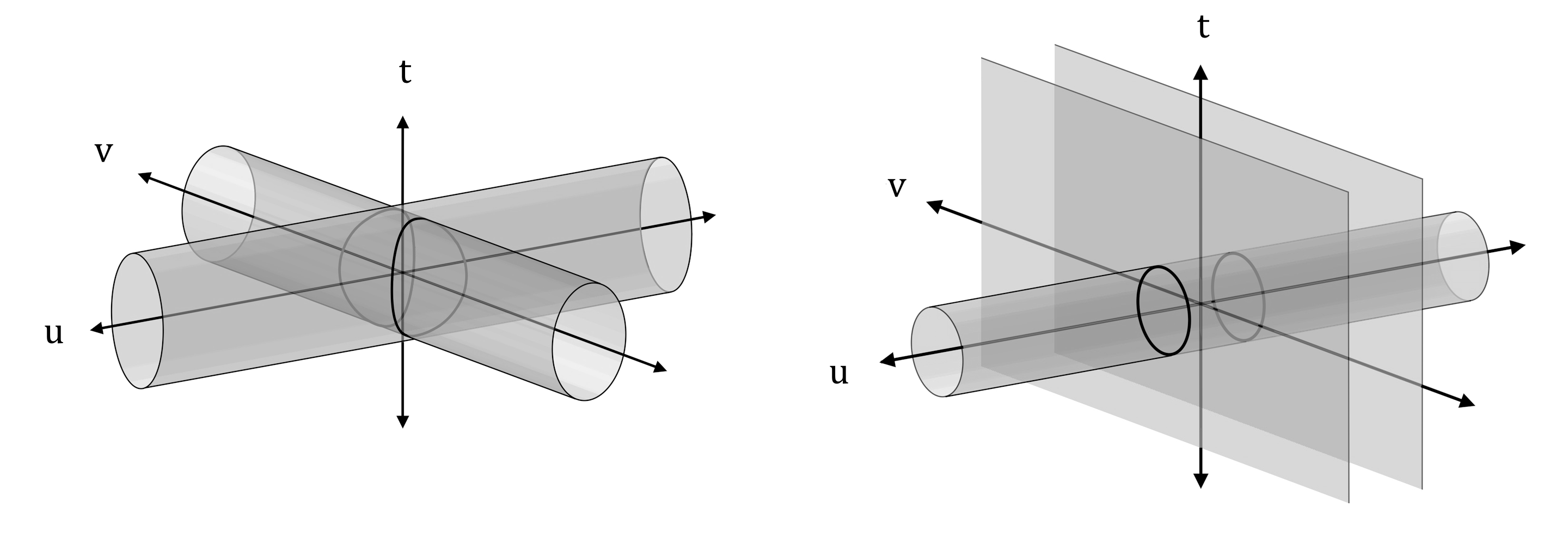} 
 \end{center}
 \caption{For the figure on the left, one cylinder is the set of conformally inscribed parallelograms whose $AC$ diagonal has squared length $1/2$ and the other consists of the parallelograms whose $BD$ diagonal has squared length $1/2$. The  points on the intersection of the two cylinders are the unit rectangles in $\Pi$, and hence the space curves formed by this intersection lie on the unit sphere.  On the right, the two planes represent the extremal situation in which $\mu=0$.}
\end{figure}

 The cylinders in Theorem~\ref{cylinders} are identical up to rotation if and only $\lambda = \mu$, if and only if the parallelogram $T$ is a rectangle.   As we show in the next theorem, this condition is significant for the geometry of the set of rectangles inscribed in  the configuration $\C$ because it coincides with a degeneracy condition for  $\C$. To define what is meant by ``degenerate'' here, 
we recall how in \cite[Definition~2.2]{OW2}  a pair of diagonals $E$ and $F$ are defined for the configuration $\C$. (These diagonals are lines in $\P^2$ rather that line segments.)
  
   If $A \ne B$, then the {\it diagonal} 
 $E$ of $\C$ is the line in $\P^2$ through the origin ($= C \cap D$) and $A \cap B$, which is possibly at infinity.  Otherwise, if $A = B$, then $E$ is the line through the origin that is orthogonal to $A$.  
 
If  $A \ne D$, then the {\it diagonal}
$F$ is the line in $\P^2$ through $A \cap D$ and $B \cap C$. 
Otherwise, if $A = D$, then $F$ is the line through $B \cap C$ that is orthogonal to $A$.

\begin{definition}  \label{deg} The configuration ${\bf C}$ is {\it degenerate} if the diagonals of $\C$ are perpendicular. 
\end{definition}

(We are assuming here the convention that the line at infinity for the projective closure $\P^2$ of $\R^2$ is perpendicular to every line in $\R^2$.) 

The {\it rectangle locus} for $\C$, which is discussed again in Section 6, is the set of centers of the rectangles inscribed in $\C$.  
The motivation for the terminology of degeneracy in Definition~\ref{deg} is that if neither pair $A,C$ or $B,D$ consists of parallel lines, then ${\bf C}$ is degenerate if and only if the rectangle locus is a degenerate hyperbola with one of the two lines that comprise the hyperbola possibly at infinity \cite[Section 9]{OW2}.  More generally, regardless of whether $A,C$ or $B,D$ consist of parallel lines, the set of rectangles inscribed in $\C$ may be viewed as a subspace of $\R^8$, with each vertex occupying two coordinates in $\R^8$. In this space, $\C$ is degenerate if and only if the set of rectangles is a pair of intersection lines in $\P^8$ \cite[Theorem~7.1]{OW2}.   Theorem~\ref{rect cor} shows  how degeneracy is reflected in the set of rectangles in~$\Pi$.

\begin{theorem} \label{rect cor} In the coordinates $(u,v,t)$, the set of rectangles in $\Pi$ is defined by the equation $v^2= u^2 +(\mu-\lambda)t^2.$  This set is a pair of planes defined by $v=\pm u$ if and only if the parallelogram $T$ from Theorem~\ref{ortho} is a rectangle, if and only if $\C$ is degenerate. Otherwise, 
 the set of rectangles is 
an elliptical cone whose central axis is either the $u$ axis or the $v$ axis  and whose apex is at the origin.
\end{theorem}


\begin{proof}
By Lemma~\ref{cylinders}, the set of unit rectangles is the intersection of $2u^2+2\mu t^2=1$ and $2 v^2+2\lambda t^2=1$.  Let $R$ be a  rectangle in $\Pi$ with coordinates  $(u,v,t)$. If $u=v=t=0$ (in which case $R$ is the degenerate rectangle consisting of a single point, the origin), then  $v^2= u^2 +(\mu-\lambda)t^2$. Otherwise, suppose not all of $u,v,t$ are zero, so that $r:=\|R\|$ is not zero. Then $({u}{r}^{-1},{v}{r}^{-1},{t}{r}^{-1})$ is a unit rectangle and so $2u^2+2\mu t^2=r^2$ and $2 v^2+2\lambda t^2=r^2$, from which we conclude $v^2= u^2 +(\mu-\lambda)t^2$.  

Conversely, if $(u,v,t)$ is a parallelogram for which  $v^2= u^2 +(\mu-\lambda)t^2$, then with $r = (2u^2+2\mu t^2)^{\frac{1}{2}}=(2 v^2+2\lambda t^2)^{\frac{1}{2}}$, we have if $r=0$ that $u=v=t=0$, and hence $(u,v,t)$ is a degenerate rectangle, while if  $r \ne 0$, then $({u}{r}^{-1},{v}{r}^{-1},{t}{r}^{-1})$ is by Lemma~\ref{cylinders} a unit rectangle, and so $(u,v,t)$ is a rectangle, which proves the set of rectangles in $\Pi$ is defined by the equation $v^2= u^2 +(\mu-\lambda)t^2$. If $\lambda \ne \mu$,  this equation defines an elliptical cone whose central axis is either the $u$-axis or $v$-axis and whose apex is at the origin. 

Now suppose $T$ is a rectangle. By Lemma~\ref{cylinders}, $\lambda = \mu$, and so the set of rectangles is defined by $v^2=u^2$ and hence is the pair of planes defined by $v=\pm u$.
By Corollary~\ref{w cor}, 
the set of parallelograms  inscribed in $\C$ is a plane in $\Pi$ that does not go through the origin. Therefore, this plane intersects the pair of planes defined by $v=\pm u$ in a pair of lines, and so $\C$ is degenerate by \cite[Theorem~7.1]{OW2}. 
Conversely, if $\C$ is degenerate, this same reference implies  the plane of parallelograms inscribed in $\C$ intersects the surface defined by $v^2= u^2 +(\mu-\lambda)t^2$ in a pair of lines. Since a plane intersects an elliptical cone in a pair of lines if and only if the plane goes through the apex, this implies  the surface defined by $v^2= u^2 +(\mu-\lambda)t^2$ is not an elliptical cone, and thus $\lambda = \mu$. Therefore, $T$ is a rectangle. 
\end{proof}


\section{Conformal solutions}

By Lemma~\ref{cylinders}, the set of unit rectangles is an algebraic curve in $\Pi$, a bicylindrical curve, that lies on the unit sphere in $\Pi$. 
In this section we show the set  of unit rectangles is a union of two simple closed analytic curves, either one of which can be taken as a conformal solution for $\C$ in the following sense.



 
\begin{definition} 
 A {\it conformal solution} for ${\bf C}$ is   
 a set ${\mathcal{R}}$  of rectangles  in $\Pi$ such that each rectangle $R$ inscribed in $\C$ and   at infinity for $\C$ 
 is scaled from a rectangle $R'$ in ${\mathcal{R}}$, 
and for all but at most one such rectangle $R $, there is a unique rectangle $R'$ in 
 ${\mathcal{R}}$ that scales to $R$. 
 
 \end{definition} 
 
 Thus a conformal solution is a  solution for the rectangle inscription problem up to scale that includes also  rectangles at infinity. The last condition that  permits two rectangles in ${\mathcal{R}}$ to scale to  the same inscribed rectangle  $R$
 for at most one choice of  $R$ is needed when dealing with a singularity that occurs in the degenerate case. This is explained by Corollary~\ref{equiv}.

 The symmetry of the set of unit rectangles  about the origin  implies that 
 if ${\mathcal{R}}$ is a conformal solution for ${\bf C}$, so is $-{\mathcal{R}} =\{-R:R \in {\mathcal{R}}\}$. With the coordinates $(u,v,t)$, finding a conformal solution is a straightforward matter of parametrization, as we show in the proof of the next theorem. 
 
 
 
  \begin{theorem} \label{locus}
  There is a  simple closed analytic curve ${\mathcal{R}}$  in $\Pi$ such that ${\mathcal{R}}$ is a 
 conformal solution for $\C$ and  ${\mathcal{R}} \cup -{\mathcal{R}}$ is the set   of unit rectangles in $\Pi$ $($see Figure $6)$.   %

 \end{theorem}

 \begin{proof}
 By Lemma~\ref{cylinders},
 the set of unit rectangles is the intersection of the  surfaces  defined by $2u^2+2\mu  t^2=1$ and $2v^2+2\lambda t^2=1$. 
We consider two cases in order to define ${\mathcal{R}}$.

Suppose first 
 $\lambda \geq \mu$. Since $\lambda +\mu =1$, this implies 
  $\lambda >0$.
  At the end of the proof we explain the modification needed if $\lambda < \mu$.
 A parallelogram in $\Pi$ with coordinates $(u,v,t)$ is on the intersection of these two quadrics if and only if $ 2v^2+2\lambda t^2=1$ and $(1-2v^2)\mu = ({1}-2u^2)\lambda$; if and only if  
 $$t=\pm \sqrt{\frac{{1}-2v^2}{2\lambda}}  \: {\mbox{ and }} \:
 u = \pm \sqrt{
 \frac{\lambda - {\mu}({1}-2v^2)}{2\lambda }}.$$
 Since $2v^2 + 2\lambda t^2 =1,$ we have $\sqrt{2}\: |v| \leq {1}$. Substituting $v = \frac{1}{\sqrt{2}}\sin(\theta)$, we obtain
 that the intersection of the two quadrics is the set of parallelograms with coordinates 
 $$
 \left( 
 \pm \frac{1}{\sqrt{2 \lambda}}\sqrt{\lambda - {\mu}\cos^2(\theta)},\: 
 \pm\frac{1}{\sqrt{2}}\sin(\theta),\:
 \pm\frac{1}{\sqrt{2\lambda}}\cos(\theta)\right).
 $$
Define $$\Phi:[0,2\pi] \rightarrow \Pi:
\theta \mapsto \left( 
 \frac{1}{\sqrt{2 \lambda}}\sqrt{\lambda - {\mu}\cos^2(\theta)},\:  
\frac{1}{\sqrt{2}}\sin(\theta),\:
\frac{1}{\sqrt{2\lambda}}\cos(\theta)\right).$$
Since $\lambda \geq \mu$, $\Phi$ is defined and analytic on all of $[0,2\pi]$.  
Using the fact that $\sin(-\theta) = -\sin(\theta)$ and $\cos(-\theta) = \cos(\theta)$, it follows that the intersection of the two quadrics (which is the set of unit rectangles) is the union of the image of $\Phi$ and the image of $-\Phi$.   The maps $\Phi$ and $-\Phi$ define simple closed analytic curves ${\mathcal{R}} = $ Image($\Phi$) and $-{\mathcal{R}}= $ Image($-\Phi$) on the unit sphere, 
and the union of these two curves  is the set of unit rectangles. 

Still in the case $\lambda \geq \mu$, to see that 
each rectangle inscribed in $\C$ or at infinity 
 is scaled from a rectangle in ${\mathcal{R}}$, let $R$ be a rectangle in $\Pi$ that is either inscribed in $\C$ or at infinity.  
 Then 
 the coordinates $(u,v,t)$ of $R$  are not all zero. Let $r = \| R\|=u^2+v^2+t^2$. Now $R$ is in the image of $\Phi$ if and only if $v \geq 0$, while $R$ is in the image of $-\Phi$ if and only if $v \leq 0$.  
 If $v \geq 0$, then $(\frac{u}{r},\frac{v}{r},\frac{t}{r})$ gives the coordinates of a unit rectangle $R'$ in the image of $\Phi$ such that $rR' = R$.   Otherwise, if $v \leq 0$, then 
  $(-\frac{u}{r},-\frac{v}{r},-\frac{t}{r})$ gives the coordinates of a unit rectangle $R'$ in the image of $-\Phi$ such that $rR' = R$.    A similar argument shows that each rectangle inscribed in $\C$ or at infinity 
 is scaled from a rectangle in $-{\mathcal{R}}$.
 
Next, we show that
  for all but at most one rectangle $R $ in $\Pi$, there is only one rectangle $R'$ in 
 ${\mathcal{R}}$ that scales to  $R$.
 Suppose two rectangles $R_1,R_2$  in ${\mathcal{R}}$
   scale to the same rectangle $R$ inscribed in $\C$ or at infinity. 
Since the image of $\Phi$ lies on the unit sphere and in the hemisphere  for which $u \geq 0$, this implies that $R_1$ and $R_2$ are the parallelograms $(0,0,1)$ and $(0,0,-1)$, both of which scale to the same parallelogram inscribed in $\C$ or at infinity for $\C$, in this case, $R$. This proves the theorem.

 If instead $\mu > \lambda$, 
  we switch the roles of the $u$ and $v$ and apply the same argument. 
  Specifically, we use the fact that a parallelogram with coordinates $(u,v,t)$ is on the intersection of the two cylinders if and only if 
   $$t=\pm \sqrt{\frac{{1}-2u^2}{2\mu}}  \: {\mbox{ and }} \:
 v = \pm \sqrt{
 \frac{\mu - {\lambda}({1}-2u^2)}{2\mu }}.$$
 %
  Using this as a starting point, the construction of $\Phi$ now proceeds as before, with the roles of $u$ and $v$ switched from the original argument. 
  \end{proof}

\begin{figure}[h] \label{cylindersgraphic}
 \begin{center}
 \includegraphics[width=0.65\textwidth,scale=.03]{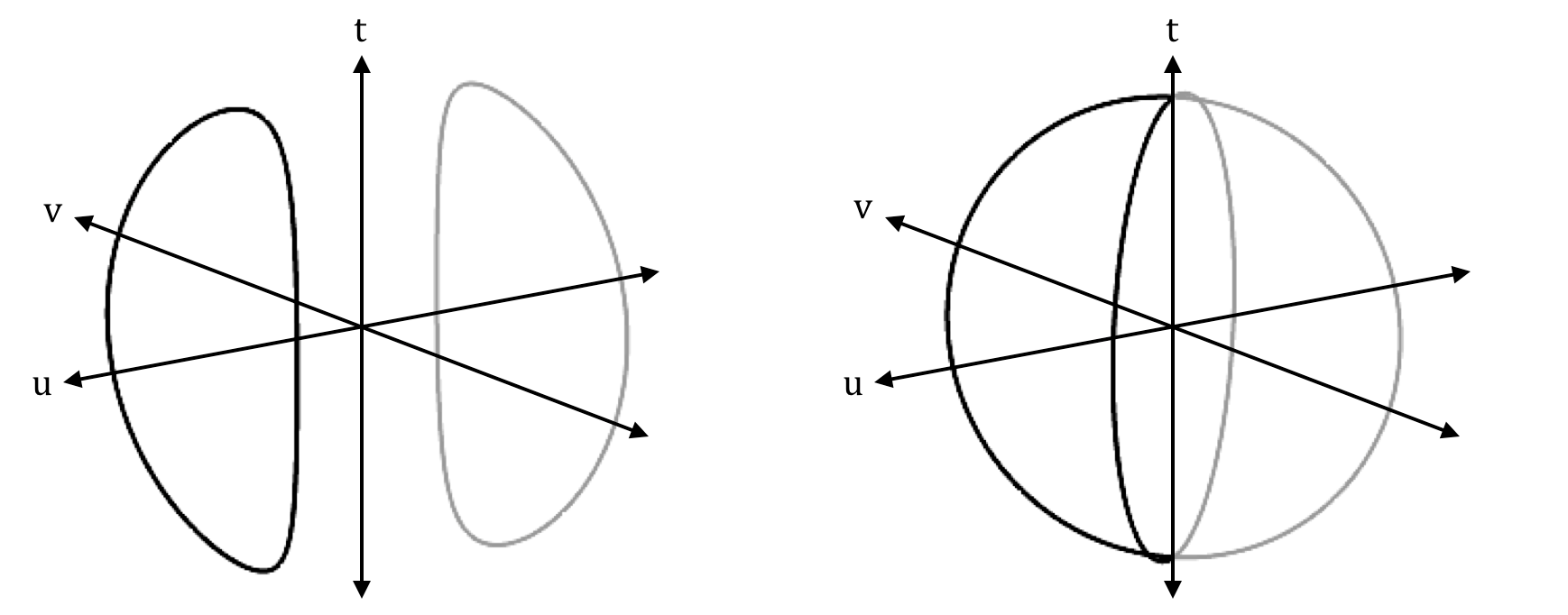} 
 \end{center}
 \caption{
 On the left is the pair of conformal solutions  for a non-degenerate configuration $\C$, one in black and the other gray. Each rectangle inscribed in $\C$ is scaled from a rectangle on the conformal solution, as is each rectangle at infinity. The figure at right is the pair of conformal solutions for a degenerate configuration. The conformal solutions in this case are not the irreducible algebraic curves, the  planar circles, but instead the analytic curves pieced from  halves of each circles. 
 The gray and black curves meet at the rectangles $T$ and $-T$. 
 }
\end{figure}

Theorem~\ref{locus} can be carried further to assert that the conformal solutions in the theorem give continuous flows of rectangles, where by a flow we mean a map $f:  {\mathcal{R}}  \times \R \rightarrow {\mathcal{R}}$ such that for all $R \in {\mathcal{R}}$ and $\theta, \rho \in \R$, we have $ f(R,0) = R$ and $f(f(R,\theta),\rho) =f(R,\theta+\rho).$

\begin{corollary} \label{flow}
There is a continuous flow of unit rectangles in $\Pi$ that is a conformal solution for $\C$ and
passes through  rectangles  at infinity for $\C$.  
\end{corollary}

\begin{proof}
By Theorem~\ref{locus}, there is a conformal solution ${\mathcal{R}}$ for $\C$ consisting of unit rectangles, and this solution is a simple closed analytic curve $\Phi:[0,2\pi] \rightarrow {\mathcal{R}}$. The proof of Theorem~\ref{locus} shows that $\Phi(\theta) = \Phi(\theta+2\pi)$. For each $\theta \in \R$, let $R_\theta = \Phi(\theta)$.  
Define a map $f: {\mathcal{R}} \times \R \rightarrow {\mathcal{R}}$ by $f(R_\theta,\rho) = R_{\theta+\rho}$. Then $f(R_\theta,0) = R_\theta$ and $f(f(R_\theta,\rho),\tau) = R_{\theta+\rho+\tau}$, so that $f$ gives a flow. The flow  $f$ is continuous since $\Psi$ is continuous. 
\end{proof}

\begin{corollary}  \label{equiv}
Let ${\mathcal{R}}$ be the  conformal solution   given by Theorem~\ref{locus}.
Then each rectangle inscribed in $\C$ or  at infinity is scaled from a unique rectangle in ${\mathcal{R}}$ if and only if 
the two conformal solutions ${\mathcal{R}}$ and $-{\mathcal{R}}$  have no rectangles in common, if and only if $\C$ is not degenerate.

   

 \end{corollary} 
 
 \begin{proof}  
 We prove the theorem for the case $\lambda \geq \mu$. If instead $\mu > \lambda$, then as in the proof of Theorem~\ref{locus}, we switch the roles of $u$ and $v$ and argue similarly. 
Let $\Phi:[0,2\pi] \rightarrow \Pi$ be the map defined in the proof of Theorem~\ref{locus} for the case $\lambda \geq \mu$.  
  Suppose two rectangles  in the image of $\Phi$ scale to the same rectangle inscribed in $\C$ or at infinity. As in the proof of Theorem~\ref{locus}, this implies the two rectangles have coordinates $(0,0,\pm 1)$, and hence these two rectangles are $T$ and $-T$.  By Theorem~\ref{rect cor}, 
  $\C$ is degenerate.
  Conversely, if $\C$ is degenerate, then by Theorem~\ref{rect cor}, $\lambda = \mu$, and so the parallelograms with coordinates $(0,0,\pm {1})$ are rectangles by  Theorem~\ref{rect cor}. These two rectangles can both be scaled to a rectangle that is inscribed in $\C$ or at infinity. 


Next, 
suppose ${\mathcal{R}}$ and $-{\mathcal{R}}$ intersect. There are $\theta_1,\theta_2 \in [0,2\pi]$ such that $\Phi(\theta_1) = \Phi(\theta_2)$, and, comparing the first coordinates of these points, we have $$\sqrt{\lambda-\mu\cos^2(\theta_1)} = -\sqrt{\lambda-\mu\cos^2(\theta_2)},$$ which forces $\lambda-\mu\cos^2(\theta_1) = 0$, and so $\lambda \leq \mu$. But we have assumed $\lambda \geq \mu$, so 
we must have $\lambda = \mu$, 
which by Theorem~\ref{rect cor} shows  $\C$ is degenerate. 
  Conversely, if 
  $\C$ is degenerate, then 
  $\lambda = \mu$ and $\Phi(0) = -\Phi(\pi)$, so the two curves intersect. This proves ${\mathcal{R}} \cap -{\mathcal{R}} \ne \emptyset$ if and only if $\C$ is degenerate. 
\end{proof}

The proof shows that if $\C$ is degenerate, then the parallelogram $T$ from Theorem~\ref{ortho} is a rectangle and  ${\mathcal{R}} \cap -{\mathcal{R}} = \{T,-T\}$.

\section{The center map}

So far, our focus has been on the Euclidean space $\Pi$ of parallelograms conformally inscribed in $\C$ and the space curve  of the unit rectangles in $\Pi$. In this case, the dimension of $\Pi$ reflects the three degrees of freedom for specifying parallelograms in $\Pi$, namely the coordinates $u,v,t$ for the orthonormal basis $U,V,T$. The advantage of this representation is that it induces  axes for $\Pi$, ``principal axes,'' that make the geometry of $\Pi$   easier to describe. 
The disadvantage    is that the coordinates $u,v,t$ are less explicitly connected to
the coordinates that express 
 the original four lines   of the configuration $\C$ and the vertices   inscribed in these lines. We work next with a different representation of $\Pi$, one that involves the same coordinates  as those that define the lines $A,B,C,D$. In the next section, this will be called the cylinder model of $\Pi$. In the present section, the focus is on a map that takes parallelograms in $\Pi$ into $\R^3$.


\begin{definition} The {\it center map}
 $\Psi:\Pi\rightarrow \R^3$ is the map that sends
 a parallelogram $P$ in $\Pi$ 
  to $(x,y,w)$, where $(x,y)$ is the center of  $P$  and $w$ is the scale of~$P$.   
 
 \end{definition}

 The image of the center map depends on whether the pairs of lines $A,C$ and $B,D$ consist of parallel lines.

\begin{proposition} \label{invertible} If neither pair $A,C$ or $B,D$ consists of parallel lines, then the center map is an  invertible linear transformation. If exactly one pair is parallel, the image is a plane in $\R^3$ through the origin. If both pairs are parallel,  the image is a line through the origin.  
\end{proposition}

\begin{proof} 
Denote the centers of the parallelograms $U,V,T$ by $(u_1,u_2),$ $(v_1,v_2),$ $(t_1,t_2),$ respectively, and denote the scales of these three parallelograms by $w_U,w_V,w_T$. 
We claim first that if $P$ is a parallelogram with coordinates $(u,v,t)$, then  the center $(x,y)$ of $P$ and scale $w$ of $P$ are given by the equation
$$\begin{bmatrix} 
x \\
y \\
w\\
\end{bmatrix} = 
\begin{bmatrix} 
u_1 & v_1 & t_1 \\
u_2 & v_2 & t_2 \\
w_U & w_V & w_T \\
\end{bmatrix} 
\begin{bmatrix} 
u \\
v \\
t\\
\end{bmatrix}. 
$$
To prove this, observe first that the center map $\Psi$ is the composition of the linear transformation that sends a parallelogram $P$ with scale $w$ and vertices $(x_A,y_A) \in A,(x_C,y_C) \in C$ to the vector  $(x_A,y_A,x_C,y_C,w) \in  \R^5$ and the linear transformation that sends a vector $(x_A,y_A,x_C,y_C,w) \in  \R^5$ to the vector $(\frac{1}{2}(x_A+x_C),
\frac{1}{2}(y_A+y_C), w) \in \R^3$.  
Thus $\Psi$ is a linear transformation, and so the first claim of the proof follows from the fact that for each parallelogram $P \in \Pi$ with coordinates $(u,v,t)$ and scale $w$, $$\Psi(P) = u\Psi(U)+v\Psi(V)+t\Psi(T) = u(u_1,u_2,w_U)+v(v_1,v_2,w_V)+t(t_1,t_2,w_T).$$

If neither pair $A,C$ or $B,D$ consists of parallel lines, then each parallelogram in $\Pi$ is uniquely determined by its center and scale (see \cite[Lemma~3.2]{OW}), and so $\Psi$ is injective and hence invertible since $\Pi$ has finite dimension. 
Next, suppose the lines $A$ and $C$ are parallel. Then for each $w \in \R$,  the line $L(w)$ in $\R^2$ equidistant from both of the parallel lines $A(w)$ and $C(w)$ must contain the midpoint of any line segment joining $A(w)$ and $C(w)$. Thus the centers of the parallelograms in $\Pi$ of scale $w$ must all lie on $L(w)$. Similarly if  $B$ and $D$ are parallel,  the centers of all parallelograms of scale $w$  must lie on a line $L'(w)$ equidistant from $B(w)$ and $D(w)$.  If both pairs $A,C$ and $B,D$ are parallel,   all the parallelograms in $\P$ of scale $w$ have their centers on the intersection of $L(w)$ and $L'(w)$, and so the image of the center map is a line through the origin. 
Otherwise, 
 if only the pair $A,C$ has parallel lines, then the image of the center map is the plane $L = \bigcup_{w \in \R}L(w)$, while if only the pair $B,D$ has parallel lines,  the image is the plane $L' = \bigcup_{w \in \R}L'(w).$ 
\end{proof}

In the next section we will work with a pair of cylinders that play roles similar to those of the cylinders 
 in $\Pi$ defined in Section~3. As shown in Proposition~\ref{old}, these cylinders are often the image of the cylinders in $\Pi$ under the center map.

\begin{definition} 
The {\it $AC$ cylinder} is the set of  points $(x,y,w) \in \R^3$ for which  $(x,y)$ is the midpoint of a line segment of squared length $1/2$ 
 joining $A(w)$ and $C(w)$. 
The $BD$ cylinder  is defined analogously. 
\end{definition}

In the next proposition, we show the $AC$ and $BD$ cylinders are indeed cylinders, possibly flat, where by a  
  {\it flat cylinder}, we mean the set of points $(x,y,w)$ in a plane $\c P$ for which there exist $b,d \geq 0$ with  $d|w| \leq b$. If $d=0$, the flat cylinder is the entire plane $\c P$ (a flat cylinder of ``infinite radius'').

\begin{proposition} \label{old}
If $A$ is not parallel to $C$, then the $AC$ cylinder in the coordinates $(x,y,w)$ is an elliptical cylinder; otherwise, 
the $AC$ cylinder is a flat cylinder. In either case, if $B$ is not parallel to $D$,  the $AC$ cylinder is the image under the center map of the cylinder in $\Pi$ defined by $v^2+2\lambda t^2=1$ $($see Lemma~$\ref{cylinders}).$   
 %
 The analogous statement holds for $B$ and $D$ in place of $A$ and $C$ and 
 $u^2+2\mu t^2=1$ in place of $v^2+2\lambda t^2=1$.
 
\end{proposition}

\begin{proof} It is enough to prove the theorem for the lines $A$ and $C$ since the proof for $B$ and $D$ is simply a matter of replacing $A$ with $B$ and $C$ with $D$.  
Suppose first that $A $ is parallel to $ C$. 
For 
 each $w \in \R$,  let $L(w)$ 
be the line  
 in $\R^2$ equidistant from  the parallel lines $A(w)$ and $C(w)$. As in the proof of Proposition~\ref{invertible}, every  midpoint of  a line segment joining $A(w)$ and $C(w)$ lies on $L(w)$, and so the $AC$ cylinder is 
 a subset of the plane 
   ${\c P}=  \bigcup_{w \in \R}L(w)$.  
 Let
  $d$ be the  distance between $A$ and $C$. 
 We claim  the $AC$ cylinder is the flat cylinder consisting of the points $(x,y,w) \in \c P$ for which $\sqrt{2}|w|d \leq 1$. 
 If  $(x,y,w) $ is a point on the $AC$ cylinder, then
  the distance between $A(w)$ and $C(w)$ is $d|w|$, and so since $(x,y) \in L(w)$ and 
  $(x,y)$ is the midpoint of a line segment of squared length $1/2$ joining $A(w)$ and $C(w)$,  it follows that  
  $\sqrt{2}|w|d \leq 1$. 
On the other hand, if $(x,y,w) \in \c P$   and $\sqrt{2}|w|d \leq 1$, then $(x,y) \in L(w)$ and since the distance between $A(w)$ and $C(w)$ is $|w|d \leq \frac{1}{\sqrt{2}}$, there is a line segment  joining $A(w)$ and $C(w)$ that has midpoint $(x,y)$ and squared length $1/2$, so that $(x,y,w)$ is on the $AC$ cylinder.  This proves that if $A \parallel C$, then the $AC$ cylinder is a flat cylinder.

Next, assume $A$ is not parallel to $C$. 
  Let $M$ be the $2 \times 3$ matrix 
defined by 
$$M=\frac{1}{m_A-m_C} \begin{bmatrix} 
-2(m_A+m_C) & 4 & -2b_A \\
 -4m_Am_C & 2(m_A+m_C) & -2b_Am_C
 \end{bmatrix}.$$ 
Let ${\bf x}=(x,y,w) \in \R^3$, and let $(x_A,y_A) \in A(w)$ and $(x_C,y_C) \in C(w)$. 
Using the fact that $y_A = m_Ax_A+wb_A$ and $y_C = m_Cx_C$, it follows that
 $(x,y)$ is the midpoint of the line segment  joining $(x_A,y_A) $ and $(x_C,y_C)$ if and only if  
$$\begin{bmatrix} 
x_A-x_C \\
y_A-y_C \\
\end{bmatrix} = M
\begin{bmatrix}
x \\
y\\
w \\
\end{bmatrix};$$
If this equation is satisfied, 
then
 $(x_A-x_C)^2+(y_A-y_C)^2 ={\bf x}^TM^TM{\bf x}$, and hence $M^TM$ is a positive semi-definite matrix. 
 Since $M^TM$ is symmetric, the principal axis theorem implies $M^TM =Q^T\Lambda Q$, 
 where $Q$ is 
  a $3 \times 3$ orthogonal matrix    and $\Lambda$ is a $3 \times 3$ diagonal matrix whose diagonal entries $\lambda_1,\lambda_2,\lambda_3$ are the eigenvalues of $M^TM$. 
 Since the determinant of the submatrix of $M$ consisting of the first two columns of $M$ is $-4$,  the rank of $M$ is $2$, and so  $\rank(M^TM)=2$. 
 Thus exactly one of the $\lambda_i$ is $0$, say $\lambda_1=0$. Also, since $M^TM$ is a positive semi-definite  matrix, all of the $\lambda_i$ are non-negative. Without loss of generality, $\lambda_1=0$ and $\lambda_2,\lambda_3>0$. 
For ${\bf x}=(x,y,w) \in \R^3$, let $\overline{\bf x} = Q{\bf x}$ and write 
 $\overline{\bf x} = (\overline{x},\overline{y},\overline{w})$. 
Then  $${\bf x}^TM^TM{\bf x} = {\bf x}^TQ^T\Lambda Q{\bf x}=  \lambda_2 \overline{y}^2+\lambda_3 \overline{w}^2.$$
Thus
$ (x_A-x_C)^2+(y_A-y_C)^2  =   {1}/{2}
$ if and only if 
$\lambda_2 \overline{x}^2+\lambda_3 \overline{w}^2={1}/{2}.$
This last equation defines an elliptical  cylinder in  the coordinates $\overline{x}, {\overline{y}},\overline{w}$, and hence also in the coordinates $x,y,w$ since $Q$ is an orthogonal matrix. 
\end{proof}

\section{Cylinder model}

 The  cylinders and conformal solutions in $\Pi$ as described in the previous sections depend only on the squared lengths $\lambda$ and $\mu$ of the diagonals of the parallelogram $T$. Because of this,  many different configurations  
share the same cylinders and conformal solutions   
  in $\Pi$. For example,  Theorem~\ref{rect cor} implies all degenerate configurations 
 produce the same cylinders and conformal solutions in the coordinates $(u,v,t)$.  
 In this section we discuss a 
different way to represent the configuration $\C$ and its inscribed rectangles, one that depends much more finely on the choice of $\C$.   (See the end of the next section for a more precise statement.) 
This representation, which we call  the cylinder model for $\C$, is 
  based on the same coordinates in which the configuration $\C$ and its inscribed parallelograms reside, but with an additional dimension, a scaling dimension. 
  Whereas a conformal solution in Section 4 consists of a flow of rectangles treated as points through a space of parallelograms also treated as points, in the cylinder model a conformal solution is a flow of rectangles, vertices and all, through a space in which the lines  in the configuration $\C$ 
  reside. 
   Using this representation, we also discuss how the centers of the rectangles can be used to track the rectangles themselves through $\C$, and how the nature of the paths these centers take can be explained  as a  projection from a higher dimension. 
This representation  will also  explain when and why  more than one rectangle can share the same center.

The basic idea is to consider parallelograms inscribed in the following planes in $\R^3$ rather than  in 
lines in $\R^2$. 
 \begin{center} ${\c A}: y=m_Ax+b_Aw$, \: ${\c B}: y=m_B x+w$, \: ${\c C}: y=m_Cx$, \: ${\c D}: y= m_Dx$.  
\end{center} 
 These planes pass through the origin in $\R^3$ and in the coordinates $(x,y,w)$ intersect the plane  $w=1$ in the lines $A,B,C,D$.
See Figure 6.
We say a  (planar) parallelogram $P$ in $\R^3$ is {\it inscribed in $\c A, \c B, \c  C, \c D$}  if  $P$ lies in a plane $w=w_0$ for some $w_0 \in \R$  and  its vertices  lie in sequence on these planes.  
 Writing the vertices of $P$ as   $(x_L,y_L,w) \in \c L$,
where $\c L \in \{\c A, \c B, \c  C, \c D\}$,  the points $(x_L,y_L)$ in $\R^2$ are vertices for a parallelogram inscribed in $\C(w)$, and hence a parallelogram in $\Pi$.  Conversely, every parallelogram   in $\Pi$ arises this way from a parallelogram inscribed in the four planes. 

We refer to the parallelograms inscribed in these four planes as  the {\it cylinder model for ${\bf C}$} because of the role the $AC$ and $BD$ cylinders from Section 5 play in the description of the rectangles for this model. 
In the next section we compare this model  with what we call the cone model for $\C$. 
To describe the rectangles inscribed in $\C$ and at infinity, it is enough, up to scale, to work in the cylinder model and describe the unit rectangles inscribed in the four planes, since these all scale to the rectangles inscribed in $\C$ or at infinity. We describe how to do this next.

\begin{figure}[h] \label{ellipsoid}
 \begin{center}
 \includegraphics[width=0.35\textwidth,scale=.03]{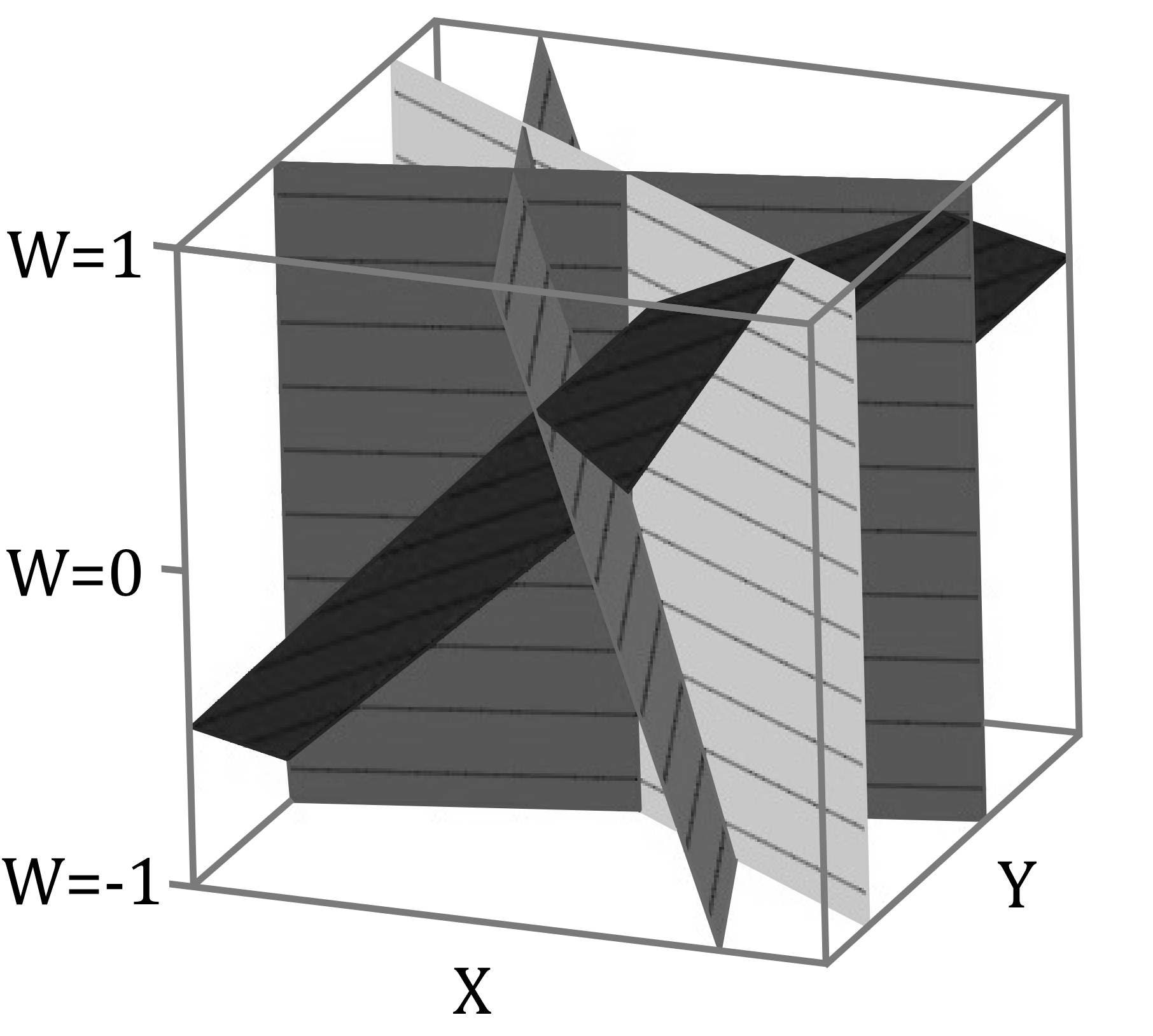} 
 \end{center}
 \caption{The four planes $\c A, \c B, \c C, \c D$.  These planes meet the  plane $w=1$ in the four lines $A,B,C,D$ and meet the plane  $w=0$ in the lines $A(0),B(0),C(0),D(0)$ that give the configuration $\C$ at infinity. The cylinder model for $\C$ consists of the (horizontal) parallelograms inscribed in these planes.
 }
\end{figure}

By the  {\it unit rectangle locus}, we mean the set of centers of the unit rectangles inscribed in $\c A,\c B,\c C, \c D$. 
The unit rectangle locus in the cylinder model is the intersection of the $AC$ and $BD$ cylinders, and thus this locus is an algebraic curve that is the union of two connected components, each the image under the center map from the last section  of a conformal solution for $\C$. 
The unit rectangle locus passes through the centers of the rectangles at infinity in the plane $w=0$. See Figures~8 and~9.

\begin{figure}[h] \label{ellipsoid}
 \begin{center}
 \includegraphics[width=0.95\textwidth,scale=.03]{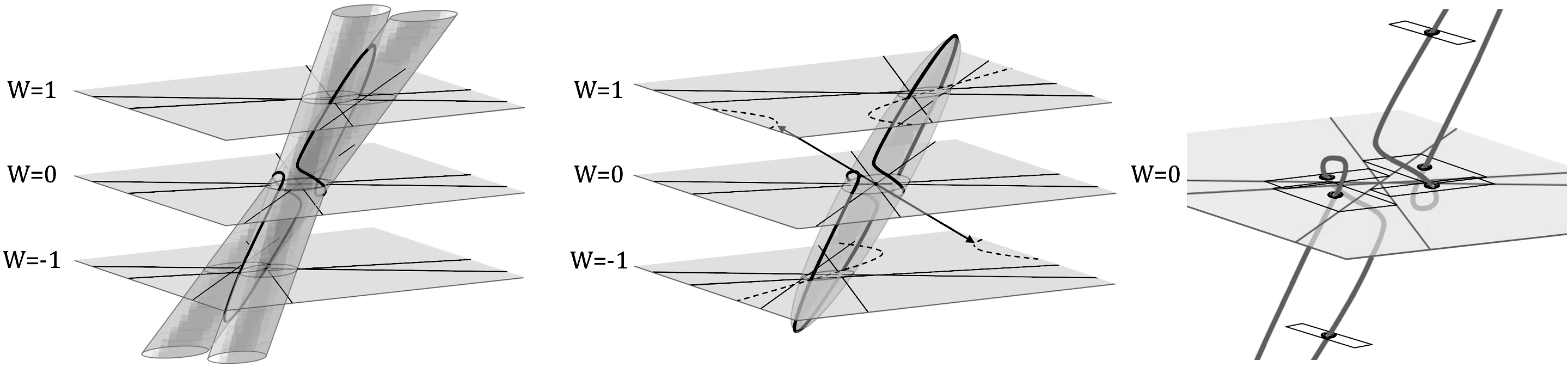} 
 \end{center}
 \caption{The elliptical cylinders at left are the images of the cylinders in Theorem~\ref{cylinders} under the center map from $\Pi$ to $\R^3$.
 The intersection (in black) of the two cylinders is the unit rectangle  locus, which lies on the ellipsoid in the middle figure. This ellipsoid is the set of centers of the unit parallelograms. The plane $w=1$ contains the original four lines that define the configuration, and the
  dotted hyperbola in this plane is the rectangle  
locus for $\C$.   Each point on the unit rectangle locus with $w \ne 0$  projects from the origin to the center of a rectangle inscribed in $\C$, as shown in the middle figure, while those in the plane $w=0$ are the rectangles at infinity for $\C$.  
The unit rectangles travel the unit rectangle locus in the figure at right and project  to the inscribed rectangles in the plane $w=1$. As shown in this same figure, the unit rectangles pass through the four unit rectangles at infinity in the plane $w=0$. Each loop consists of the set of centers of a conformal solution for~$\C$.}
\end{figure}

\begin{figure}[h] \label{ellipsoid}
 \begin{center}
 \includegraphics[width=0.85\textwidth,scale=.03]{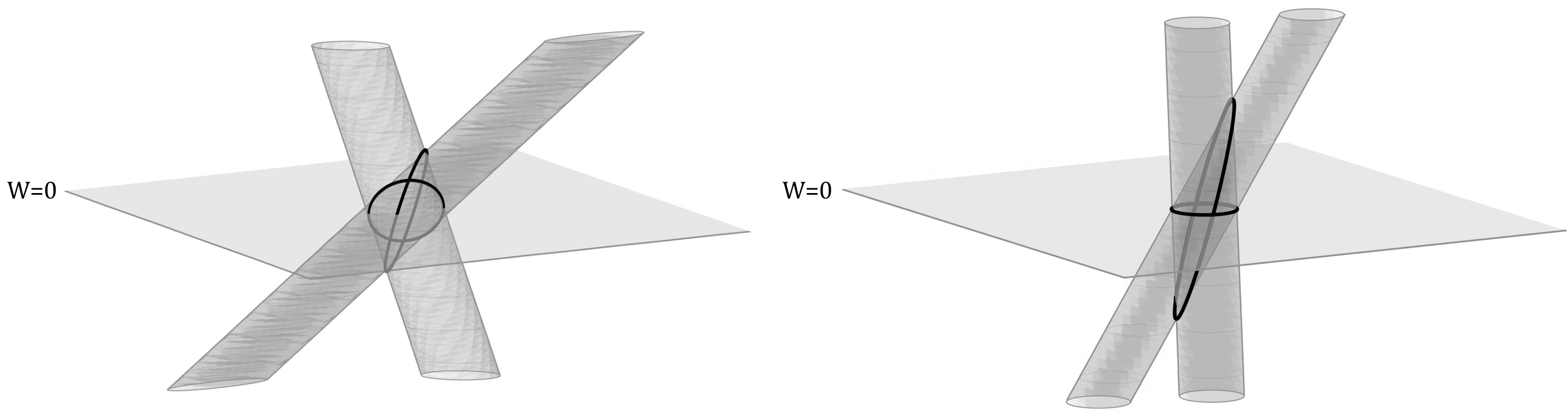} 
 \end{center}
 \caption{The figure at left shows the unit rectangle locus for a degenerate configuration. The two simple closed  smooth curves in the figure meet at the centers of the unit rectangles $T$ and $-T$. A conformal solution traces one of the  curves until reaching $T$, after which it traces the other curve until reaching $-T$, where it again switches to the original  curve.    
The figure at right shows the unit rectangle locus in the case in which $A$ is orthogonal to $C$ and $B$ is orthogonal to $D$.  This also represents a degenerate configuration. A conformal solution in this case traces out a path that travels the vertical loop for $w \geq 0$ and, upon reaching the plane $w=0$, travels half the loop in $w=0$ then remerges again on the vertical loop with $w \geq 0$. The other conformal solution traces the path that is the reflection through the origin of the first path.}
\end{figure}

Off the plane $w =0$ in the cylinder model, the unit rectangle locus projects  to the {\it rectangle locus} in the plane $w=1$,  the curve   traced by the centers of the rectangles inscribed in $\C$.
This locus in the plane has been studied in 
  \cite{OW,OW2,Sch,Sch2,Tup}.
If some of the pairs of lines among $A,B,C,D$ are parallel or perpendicular, it can happen that the rectangle locus is a point, line or a line with a missing segment; see \cite[Section~4]{OW} for a more precise statement. In \cite[Section~9]{OW2}, these   behaviors are explained by showing that
 if the locus is not itself a hyperbola, it is a ``flat hyperbola,'' that is, 
the image of a  hyperbola  under a rank one affine transformation. 
See Figure~10 for another explanation of this behavior. 

 Otherwise, 
    if none of the lines $A,B,C,D$ are parallel or perpendicular to each other, the rectangle locus is a hyperbola that is degenerate if and only if the configuration $\C$ is degenerate; if and only if the diagonals of $\C$ are perpendicular; see \cite[Theorem~7.1]{OW2} and \cite[Theorem 3.3]{Sch}. In this case, if $\C$ is degenerate,    the rectangle locus is a pair of distinct and intersecting lines, and  the parallelogram $T$ from Theorem~\ref{ortho}
 is a rectangle by Theorem~\ref{rect cor}. The proof of Corollary~\ref{equiv} implies this rectangle scales to the 
 unique rectangle\footnote{The slope of this rectangle is the same as  the slope of one of the rectangles at infinity and the aspect ratio of the rectangle is the aspect ratio of a different rectangle at infinity; see \cite[Section~7]{OW2}.}  inscribed in $\C$ whose center 
is the intersection of the two lines.

The second example in Figure 9 is a good illustration of how the cylinder model clarifies the behavior of the rectangle locus for $\C$.  In this case, $A \perp C$ and $B \perp D$, and the rectangle locus is a line in $\R^2$. As discussed in \cite[Corollary 7.3]{OW2}, the line at infinity for $\R^2$ gives another line of degenerate rectangles, all at infinity. Figure 9 shows how in the cylinder model, this exceptional case is not so exceptional after all: the unit rectangle locus behaves as the rectangle locus in a generic degenerate configuration, only with one of the simple closed curves lying in the plane $w=0$.   This phenomenon of having infinitely many rectangles inscribed at infinity is characterized in \cite[Theorem~7.3]{OW2}, where this is shown to occur if and only if $\C$ consists of ``twin pairs;'' see \cite[Definition~3.3]{OW2} for the definition.

It is possible that more than one rectangle in $\Pi$ shares the same center in the cylinder model. By Proposition~\ref{invertible}, this only happens if at least one of the pairs $A,C$ or $B,D$  consists of parallel lines. If $A$ and $C$ are parallel, then the $AC$ cylinder is flat by Proposition~\ref{old}. Proposition~\ref{invertible} suggests viewing the  $AC$ cylinder in this case as ``two points thick'' everywhere except the boundary of this flat cylinder, where the cylinder is ``one point thick.'' To determine how many rectangles inscribed in $\C$ share a particular center $p$ is then simply a matter of taking a point $p'$ on the unit rectangle locus that projects to $p$ and 
multiplying the thickness  at $p'$  
of one cylinder by the thickness of the other cylinder at $p'$.  
It follows that  at most four rectangles inscribed in $\C$ share the same center, and this only if $A$ is parallel to $C$ and $B$ is parallel to $D$.

\begin{figure}[h] \label{ellipsoid}
 \begin{center}
 \includegraphics[width=0.85\textwidth,scale=.03]{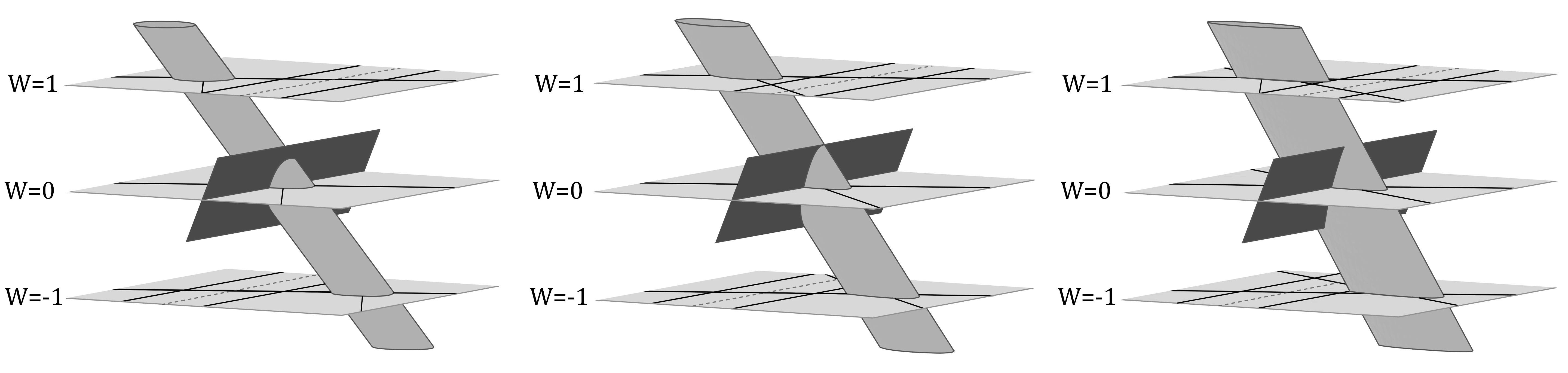} 
 \end{center}
 \caption{
Each of these figures represents a slightly different configuration in which the lines $A$ and $C$ are parallel and remain fixed. The dark plank is the $AC$ cylinder, which in this case is flat. The elliptical cylinder is the $BD$ cylinder, which is a different cylinder in each image since the lines $B,D$ are different for each image. The unit rectangle locus is the intersection of these two cylinders.  
 The figures at left and right are nondegenerate configurations. In the first figure, the unit rectangle locus is a planar ellipse while in the third figure it is a piece of a planar ellipse. In the former, the rectangle locus in the plane $w=1$ is a line, while in the third, a line with a missing segment. The middle figure is a degenerate configuration, which is evidenced by the fact that the boundaries of the two cylinders intersect.
 }
\end{figure}


\section{Cone model}

Recall from the last section that the rectangle locus for the configuration $\C$ is the set of centers in $\R^2$ of the rectangles inscribed in $\C$.  
In \cite[Lemma 4.2]{OW}, the rectangle locus for  $\C$ is described  as the projection to the plane of the intersection of a pair of surfaces in $\R^3$ associated to the pairs $A,C$ and $B,D$. Each of these surfaces is a cone, in an expanded sense of the word by which we mean either   a real elliptical cone or a {\it flat cone}:  a vertical plane in $\R^3$, possibly with a missing midsection (see \cite[Section 3]{OW}). 
 The former occurs if the lines in the pair are not parallel and the latter if they are. 
 
 We call this
 setting, which consists of $\R^3$, these cones and their intersections,  the {\it cone model} for~$\Pi$. 
In this section we  show how the cone and cylinder models can be obtained from each other by viewing each model as a subspace of $\P^3$. In fact, the difference between these two models is simply a matter of where the plane at infinity in $\P^3$ is placed. 

We denote the homogeneous coordinates in $\P^3$ by  $[x:y:z:w]$, and we use $(x,y,z)$ to designate the points in the cone model and  $(x,y,w)$   the points in the cylinder model. 
The cone model resides in the affine chart $w \ne 0$ of $\P^3$
 and the cylinder model   in the affine chart $z\ne 0$.  
The affine plane in $\P^3$ defined by $z=w=1$ contains the original configuration $\C$ of four lines and the parallelograms and rectangles inscribed in $\C$. The line at infinity for this plane is the line in $\P^3$ defined by $z=w=0$. Put another way:
  {\it The
plane $w=0$ in the cylinder model is the plane at infinity for the cone model, and the plane $z=0$ in the cone model is the plane at infinity for the cylinder model.}

This observation makes  heuristic arguments possible. For example, a cone in the cone model with apex in the plane $z=0$ becomes a cylinder in the cylinder model, since a cylinder is a cone with its apex at infinity; see Figure~11. Similarly, the ellipsoid centered at the origin in the cylinder model becomes a hyperboloid in the cone model.

\begin{figure}[h] \label{ellipsoid}
 \begin{center}
 \includegraphics[width=0.9\textwidth,scale=.03]{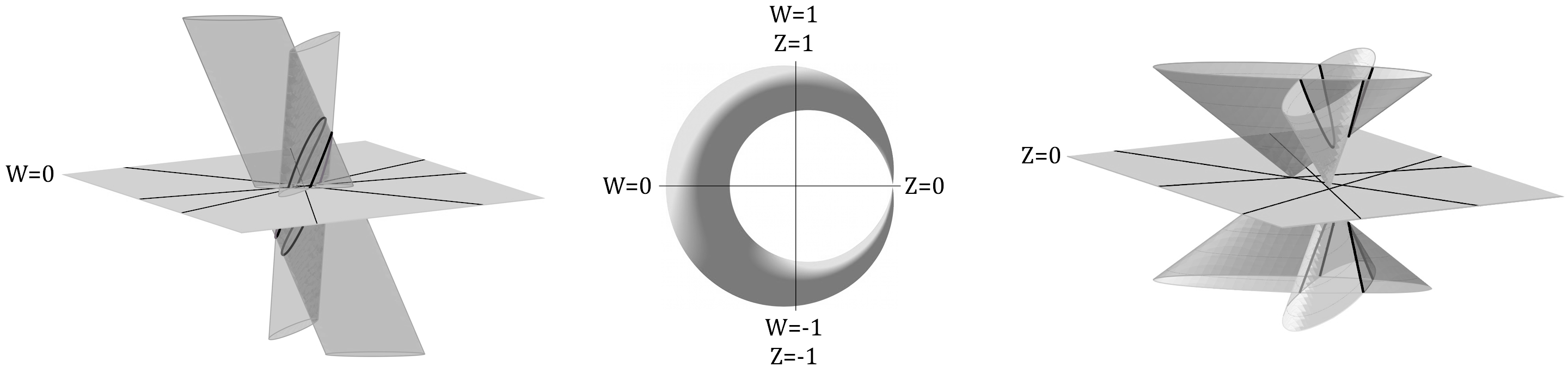} 
 \end{center}
 \caption{The cylinders in the figure at left become the cones at right under the perspective transformation that exchanges the plane $w=0$ with the plane $z=0$.  
The figure in the middle is a conceptual diagram that represents a projective object that
is a cone when $w=0$ is the plane at infinity and a cylinder when $z=0$ is the plane at infinity. 
  The original configuration of four lines resides in the plane in $\P^3$ defined by $z=w=1$. The perspective transformation leaves  this plane untouched, as it does also the plane defined by $z=w=-1$, which contains the reflection through the origin of the original  configuration.
  }
\end{figure}

For the sake of completeness, we   give algebraic arguments to indicate a few of the connections between the two models. 
 This is a matter of verifying that the projective closures of the various surfaces and curves in  the  cone model yield the corresponding surfaces and curves in the cylinder model. 
 
First, we say  that in the cone model, the {\it AC cone}  consists of the points  $(x,y,z)$ in $\R^3$  such that $(x,y)$ is the midpoint of a segment   whose endpoints lie on $A$ and $C$ and $2|z|$ is its length.
Suppose first that $A$ and $C$ are not parallel. 
It is shown in \cite[Theorem 3.3]{OW} that in this case,  the $AC$ cone is a real elliptical cone whose equation has the form
 $$(a_1x+a_2y+a_3)^2 +(b_1x+b_2y+b_3)^2=z^2,$$ 
where $a_i,b_j \in \R$.  This cone has central axis parallel to the $z$-axis and its apex is in the plane $z=0$.  The homogenization of the equation of this cone is $$(a_1x+a_2y+a_3w)^2 +(b_1x+b_2y+b_3w)^2=z^2.$$ In the chart $z\ne 0$,  
we can rewrite this  as  $$2(a'_1x+a'_2y+a'_3w)^2 +2\lambda(b'_1x+b'_2y+b'_3w)^2=1,$$ where $a'_i,b'_j \in \R$ and $\lambda$ is as in Notation~\ref{lambda}. (Since $A$ and $C$ are not parallel, $\lambda \ne 0$.)
  This is the equation of the $AC$ cylinder  in the cylinder model.   
Thus, taking the projective closure of the $AC$ cone in the chart $w\ne 0$ of $\P^3$ produces a conical surface whose apex is in the plane $z=0$. In the chart $z \ne 0$
this is the $AC$ cylinder, which can be viewed as a cone having apex at infinity (the plane $z=0$ in this case).    


Otherwise, if $A$ and $C$ are parallel, then the $AC$ cone is a flat cone, a vertical plane with a missing midsection. Thus there are $a,b,c,d \in \R$, $d\geq 0$, such that the surface is defined by   $ax+by+c=0$ and $ |z| \geq d$. 
 (See \cite[Proposition~3.8]{OW} for more details.) Taking the projective closure of this surface and  working in the chart $z\ne 0$ yields the flat cylinder from the last section, which is defined by $ax+by+cw=0$ and $dw \leq 1$.

Next, we compare how the rectangles inscribed in $\C$ are handled in each model. 
 In the cylinder model,  the set of unit  parallelograms is   an ellipsoid (possibly flat) that is symmetric about the origin, which in the cone model becomes the hyperboloid consisting  of  
 points $(x,y,z)$ such that $(x,y)$ is the center of a parallelogram in $\Pi$ with norm $|z|$.  
 On this   hyperboloid lies the curve consisting of the points $(x,y,z)$ such that $(x,y)$ is the center of a rectangle of norm $|z|$.  This curve consists of two branches for $z>0$ and two branches for $z<0$.  Since the plane at infinity for the cone model is $w=0$, it follows that in the cylinder model these branches become the pair of closed curves on the ellipsoid that define the unit rectangle locus. In the cylinder model, both the rectangles inscribed in $\C$ and at infinity appear, while in the cone model the rectangles at infinity are absent. 
 
 Finally, we note that in almost all cases, the pairs of 
lines $A,C$ and $B,D$ that define  
  the cones and cylinders in these two models can in turn be read off from the shapes of the cones and cylinders themselves, thus allowing either model to stand in place of the original rectangle inscription problem. To explain this, we first recall what distinguishes $AC$ and $BD$ cones among elliptical cones. Namely, the cones that are not flat and arise as $AC$ or $BD$ cones are {\it hyperbolically rotated}  from circular cones (more precisely, from ``unit cones''); see \cite[Section 2]{OW}.  Such cones, when elliptical with  vertical central axis and apex at $z=0$, 
  are distinguished by the   fact that the elliptical cross section of the cone at $z=1$ 
 has area $\pi$. 
 Since the $AC$ and $BD$ cylinders are the result of a perspective transformation of the $AC$ and $BD$ cones that leaves intact the affine  plane in $\P^3$ defined by $z=1$ and $w=1$, the $AC$ and $BD$ are cylinders are hyperbolic rotations of circular cylinders, but more simply, these are the cylinders whose horizontal elliptical cross sections have area $\pi$. 

Conversely, given a hyperbolically rotated 
 cone,  {\it one that is not a circular cone}, with vertical central axis and apex in the plane $z=0$, it's possible to 
find the lines $A $ and $C$ that gave rise to this cone. This is done 
 by intersecting the $AC$ cone with a circular cone sharing the same apex; see \cite[Theorem~3.6]{OW} for details on  how to do this. 
 Since the $AC$ cylinder is the $AC$ cone under the perspective transformation discussed above, it follows that the $AC$ cylinder also uniquely determines the lines $A$ and $C$. 
 The analogous statement holds for $B$ and $D$.  
 
 But what if one of the cones  is circular? In this case, the cone, and hence the corresponding cylinder,  is generated by any pair of perpendicular lines in the $xy$-plane  that intersect at the apex of the cone, and conversely, any pair of perpendicular lines gives rise to a circular cone (see \cite[Theorem~3.6]{OW}) and hence a cylinder whose horizontal cross sections are circles. 
 While this 
 might appear to be a shortcoming in the representation of the configuration $\C$ in our three-dimensional models,  it actually expresses something interesting about the case in which one of the pairs, say $A,C$, consist of perpendicular lines. In this case, any rotation of the pair $A,C$ results in the same $AC$ cone and hence the same rectangle locus.  
 Moreover, as the orthogonal pair $A,C$ rotates, the rectangles on this fixed  locus change shape as the lines rotate, but these rectangles never change size in the sense that their norm, as defined in Section 3, does not change.

 \medskip
 {\it Conflict of interest statement}. 
 On behalf of all authors, the corresponding author states that there is no conflict of interest.

\end{document}